\documentclass[11pt,reqno]{amsart}
\usepackage[T1]{fontenc}
\usepackage{lmodern}
\usepackage{amsmath,amssymb,mathtools,mathrsfs}
\usepackage{microtype}
\usepackage[a4paper,textwidth=146mm,textheight=220mm,centering,headheight=14pt,headsep=7mm]{geometry}
\usepackage{enumitem,booktabs,array}
\usepackage[hidelinks]{hyperref}
\usepackage[capitalise,nameinlink,noabbrev]{cleveref}
\hypersetup{pdftitle={Sharp quadratic lower bounds for anticanonical volumes of weak Fano surfaces},pdfauthor={Pinxian Bie},pdfsubject={Optimal quadratic volume bounds for epsilon-lc weak del Pezzo surfaces}}
\numberwithin{equation}{section}
\newtheorem{theorem}{Theorem}[section]
\newtheorem{proposition}[theorem]{Proposition}
\newtheorem{lemma}[theorem]{Lemma}
\newtheorem{corollary}[theorem]{Corollary}
\theoremstyle{definition}

\theoremstyle{remark}
\newtheorem{remark}[theorem]{Remark}
\newcommand{\eps}{\varepsilon}
\newcommand{\Q}{\mathbb Q}
\newcommand{\R}{\mathbb R}
\newcommand{\Z}{\mathbb Z}
\newcommand{\C}{\mathbb C}
\newcommand{\PP}{\mathbb P}
\newcommand{\Mcyc}{M_{\mathrm{cyc}}}
\newcommand{\cexc}{c_{\mathrm{exc}}}
\newcommand{\hzero}{h_0}
\DeclareMathOperator{\vol}{vol}
\DeclareMathOperator{\Supp}{Supp}
\DeclareMathOperator{\mld}{mld}
\DeclareMathOperator{\lct}{lct}
\DeclareMathOperator{\coreg}{coreg}
\DeclareMathOperator{\conv}{conv}
\DeclareMathOperator{\area}{area}
\DeclareMathOperator{\diag}{diag}
\DeclareMathOperator{\ord}{ord}
\DeclareMathOperator{\NCan}{NCan}
\DeclareMathOperator{\adj}{adj}

\DeclareMathOperator{\ind}{ind}
\newcommand{\one}{\mathbf 1}
\setlist[enumerate]{label=\textup{(\roman*)},leftmargin=2.1em,itemsep=3pt}
\setlist[itemize]{leftmargin=2em,itemsep=3pt}
\allowdisplaybreaks[1]
\title[Sharp quadratic bounds for anticanonical volumes of weak Fano surfaces]{Sharp quadratic bounds for anticanonical volumes of weak Fano surfaces}
\author{Pinxian Bie}
\address{School of Mathematical Sciences, Fudan University, Shanghai, China}
\email{23110180002@m.fudan.edu.cn}
\date{October 13, 2026}
\begin{document}
\begin{abstract}
We prove an explicit lower bound of optimal quadratic order for the anticanonical volume of every complex $\eps$-log canonical weak del Pezzo surface, with no restriction on its Picard number or its singularities. More precisely, the volume is at least
$\min\{\cexc,\eps^2/3032064\}$, where $\cexc>0$ is an explicit absolute constant. Quotients of the projective plane show that the exponent two is optimal. We also prove cubic upper bounds for the cyclic determinant mass and for the Picard number of a smooth rational log Calabi--Yau pair. We give geometric examples showing that these two auxiliary exponents are optimal in their respective settings.
\end{abstract}
\maketitle
\tableofcontents

\section{Introduction}\label{sec:intro}

Let $X$ be a normal projective surface over $\C$ with klt singularities and with $-K_X$ nef and big. We call $X$ a \emph{weak del Pezzo surface}; if $-K_X$ is ample, we call it a \emph{Fano surface} or \emph{del Pezzo surface}. Its anticanonical volume is
\[
 V(X):=\vol(-K_X)=(-K_X)^2>0.
\]
For a fixed $\eps>0$, boundedness gives a positive lower bound for $V(X)$ on the class of $\eps$-lc Fano surfaces. Surface boundedness was established by Alexeev \cite{Alexeev}; Birkar \cite[Theorem~1.1]{Birkar} proves boundedness in every fixed dimension. In this paper, we study the dependence of the lower bound on $\eps$ as $\eps$ tends to zero.

An effective bound for the Cartier index of $K_X$ gives an effective volume bound. The index estimates of Alexeev--Mori \cite{AM}, in the form recorded in \cite[Lemma~2.7]{LiuComp}, give a bound of the form
\[
 V(X)\ge \frac12(\eps/2)^{128\eps^{-5}}.
\]
This estimate is not polynomial in $\eps$ and far from being optimal. The obstruction is that a global index must clear all local denominators simultaneously, while one positive intersection number involves much less information.

The first version of this paper replaced the global index by local denominators and obtained a polynomial bound of order $\eps^{84}$. The second version obtained order $\eps^3/(1+\log(1/\eps))$ by combining a cubic determinant-mass estimate with contractions. The present paper proves the optimal quadratic order in the full class of weak del Pezzo surfaces. All arguments needed for the result are included here in this paper; the proofs do not use either earlier version.

Throughout, $0<\eps\le1$, and $\eps$-lc means that every divisorial log discrepancy is at least $\eps$. Set
\begin{equation}\label{eq:constants}
 \cexc:=\bigl(2\cdot84^{128\cdot42^5}\bigr)^{-1},
 \qquad \hzero:=\frac1{3032064}.
\end{equation}
The first constant comes from the effective treatment of exceptional Fano surfaces in \cite[Corollary~4.5]{LiuComp}. Its size does not affect the optimal exponent.

\begin{theorem}\label{thm:main}
Let $X$ be a complex $\eps$-lc weak del Pezzo surface. Then
\begin{equation}\label{eq:main}
 V(X)\ge \min\{\cexc,\hzero\eps^2\}
       \ge \min\{\cexc,\hzero\}\eps^2.
\end{equation}
If the anticanonical model of $X$ is nonexceptional, then $V(X)\ge\hzero\eps^2$.
The exponent two is optimal: for every $p<2$ there is no positive absolute constant $c$ such that $V(X)\ge c\eps^p$ for all such surfaces.
\end{theorem}

We do not claim that the numerical constant is optimal and assume it to be far from optimal. The point of \cref{thm:main} is that the same quadratic order holds without a bound on the Picard number, the number of singular points, or the lengths of the exceptional chains. To the best of our knowledge, this is the first lower bound of optimal quadratic order in this generality.

The proof also gives explicit estimates for two related invariants. Let $f:Y\to X$ be the minimal resolution. For a singular point $p$, let
\[
 M_p=-\bigl(E_{p,i}\cdot E_{p,j}\bigr)_{ij},
 \qquad D_p=\det M_p.
\]
We call $D_p$ the \emph{resolution determinant}. For a cyclic quotient $\frac1r(1,q)$, it is $r$. For a noncyclic quotient, it need not be the order of the local quotient group. Define the \emph{cyclic determinant mass} by
\[
 \Mcyc(X):=\sum_{\substack{p\in\NCan(X)\\p\text{ cyclic}}}D_p,
\]
where $\NCan(X)$ denotes the set of noncanonical singular points.

\begin{theorem}\label{thm:cubic}
For every complex $\eps$-lc weak del Pezzo surface,
\begin{equation}\label{eq:cubic-main}
 \rho(Y)\le2+4\eps^{-2}+12\eps^{-3}\le18\eps^{-3},
 \qquad \Mcyc(X)\le344\eps^{-3}.
\end{equation}
More generally, if $S$ is a smooth projective rational surface and $\Theta\ge0$ is a $\Q$-divisor such that $(S,\Theta)$ is $\eps$-lc and $K_S+\Theta\sim_\Q0$, then
\[
 \rho(S)\le2+4\eps^{-2}+12\eps^{-3}.
\]
The exponent three for $\Mcyc$ is optimal on Fano surfaces. The exponent three for the last Picard-number estimate is optimal on the stated class of smooth rational log Calabi--Yau pairs.
\end{theorem}

The distinction between the two Picard-number statements matters. Our examples proving cubic Picard growth are general log Calabi--Yau pairs; they do not show that cubic growth occurs for the minimal resolutions of Fano surfaces.

\subsection{The proof of the quadratic bound}
In this subsection, we briefly explain where the global geometry enters. Choose an lc $n$-complement $B$, with $n\in\{1,2,3,4,6\}$, on a non-exceptional Fano surface. On the minimal resolution write
\begin{equation}\label{eq:intro-setup}
 K_Y+\Delta=f^*K_X,\qquad
 P=-f^*K_X,\qquad
 K_Y+\Gamma=f^*(K_X+B),\qquad U=\Gamma-\Delta=f^*B.
\end{equation}
Thus $U\sim_\Q P$ is nef and big, $U^2=V(X)$, and $\Gamma$ is an effective boundary whose coefficients belong to $\frac1n\Z$.

The local calculations compare two corrections contributed by a singular point on a curve: the discrepancy correction in adjunction and the self-intersection correction in numerical pullback. Their sum has a global budget of at most $2+30V(X)$. At quadratic volume scale, this forces every strict component of $\Gamma$, meaning a component not exceptional over $X$, to be a smooth rational $(-1)$-curve meeting exactly two noncanonical exceptional fibres. A cyclic age calculation then proves that its image on $X$ has negative square. These curves can be contracted while preserving the complement and the $\eps$-lc Fano condition.

The next step concerns the entire boundary. A ruling on $Y$ gives an edge weight $1/(ab)$ for adjacent fibre components with ambient log discrepancies $a,b$. Under a node blowup its increment is exactly $h/(abc)$, where $h$ is the horizontal boundary multiplicity and $c$ is the new discrepancy. The canonical bundle formula bounds the total horizontal multiplicity used at nodes. This proves \cref{thm:cubic}. A refined count shows that the determinant mass has quadratic order if the number of vertical exceptional curves of discrepancy less than $1/6$ is bounded.

We convert that initial mass bound into a volume bound by choosing a suitable complement component at each contraction. The reciprocal-volume costs are bounded by a constant times $\Mcyc/m^3$, where $m$ is the current number of components. Since one fixed complement is retained, $m$ strictly decreases. The resulting series is convergent. The disconnected non-klt-locus theorem of Koll\'ar--Kov\'acs \cite[Proposition~5.1]{KK} then implies that, below $\hzero\eps^2$, the floor of $\Gamma$ is a connected chain or cycle with an edge.

The coregularity-zero complement theorem of Figueroa--Filipazzi--Moraga--Peng \cite[Theorem~4]{FFMP} allows us to choose a $2$-complement. Its crepant boundary has the form
\[
 \Gamma=S+\tfrac12H.
\]
Here $S$ is a connected chain or cycle. We exclude positive Iitaka dimension for the fractional support by the log surface MMP, the reduced-boundary volume bound of Liu--Shokurov \cite[Theorem~1.4]{LS}, and a rational-fibre calculation. We are left with two configurations.

If $H\ne0$, then $S$ is a chain with four attached fractional arms. When the positive-support intersection matrix is nonsingular, an integral determinant and the weighted matrix-tree identity give the required quadratic estimate. When it is singular, an integral kernel vector gives a principal divisor with controlled coefficients. Resolving its pencil supplies a rational fibration and the same order of estimate.

If $H=0$, then $\Gamma$ is a reduced anticanonical cycle. Boundary contractions either produce a complete fibre in the cycle or reduce its intersection word to one of three possibilities. A half-plane lattice estimate treats the fibre case. In the remaining cases we repeat the intersection word and its coefficients on a toric comparison surface. This compares volumes directly and completes the proof.

Thus a cubic determinant can occur, but it cannot cause a loss of a power of $\eps$ in the volume bound. The final argument uses the relation between local denominators and the global boundary, rather than a universal quadratic bound for those denominators.

\subsection{Related results and organization}
In the separate paper \cite[Section~1]{BiePolar}, we prove the bound $\eps^2/100$ for Fano surfaces of Picard number one, sharp-order results for integral polarisations in Picard number one, and polynomial estimates for polarised log Calabi--Yau surfaces. The present proof is independent of every result of that paper. In particular, we do not use its rank-one classification, its polarisation estimates, or its threshold bounds. The stronger numerical constant there concerns a smaller class; here we establish the optimal exponent for arbitrary Picard number.

The cubic estimate without a logarithmic loss in \cref{rem:pure-cubic} is an immediate consequence of the contraction argument. The quadratic bound in \cref{thm:main} also requires the global boundary analysis.

In Section~\ref{sec:prelim}, we fix notation and prove the local intersection formulas. Section~\ref{sec:fibres} proves the weighted-fibre and cubic estimates. Section~\ref{sec:curves} proves the contact and contraction results at quadratic scale. Section~\ref{sec:low} relates low discrepancies to determinant mass and volume. Section~\ref{sec:boundary} reduces to the two boundary configurations. Sections~\ref{sec:fractional} and \ref{sec:cycles} treat them and prove \cref{thm:main}. Section~\ref{sec:examples} gives the examples establishing the asserted optimal exponents.

\subsection*{Acknowledgements}
The author is grateful to his adviser Meng Chen for his great support. He thanks Peien Du and Zhengjie Yu for checking the first draft of this paper, and Minzhe Zhu for many valuable suggestions.

\subsection*{AI disclosure}
The author started working on this problem more than a year ago, before the rise of AI. He obtained several useful partial results along the way. Recently, the author used ChatGPT for research assistance, including exploratory discussions, elementary algebraic computations, locating references, and language editing. ChatGPT also assisted in restructuring this revision from the author's manuscript and subsequent research notes. The author is responsible for verifying all of the arguments and for the final content of the paper.

\section{Preliminaries}\label{sec:prelim}

\subsection{Standard Notations}
Throughout the paper, all varieties are defined over $\C$. Divisors are $\Q$-divisors unless stated otherwise. We use log discrepancies: a boundary coefficient $1-a$ has log discrepancy $a$. We use the standard definitions of klt, lc, dlt, and plt pairs as in \cite[Chapter~2]{KM}. A log Calabi--Yau pair is a projective pair $(S,\Theta)$ with $\Theta\ge0$ and $K_S+\Theta\sim_\Q0$.

\subsection{The anticanonical model and the resolution}

\begin{lemma}\label{lem:ac}
The anticanonical morphism of a klt weak del Pezzo surface is a crepant birational morphism $g:X\to X^{\rm ac}$ to a klt Fano surface. It preserves the volume and all log discrepancies.
\end{lemma}
\begin{proof}
By the base-point-free theorem \cite[Theorem~3.3]{KM}, a sufficiently divisible $-mK_X$ defines a morphism with connected fibres such that $-mK_X=g^*A$ for an ample Cartier divisor $A$ on the target. Bigness makes the morphism birational. We choose compatible canonical divisors. Pushing forward gives $-mK_{X^{\rm ac}}\sim A$, so the target canonical divisor is $\Q$-Cartier and its negative is ample.
The exceptional divisor $K_X-g^*K_{X^{\rm ac}}$ is numerically trivial over the target. Applying the negativity lemma to it and its negative \cite[Lemma~3.39]{KM}, we obtain
$K_X=g^*K_{X^{\rm ac}}$. The discrepancy and volume assertions follow.
\end{proof}

Let $f:Y\to X$ be the minimal resolution of a weak del Pezzo surface. Klt surface singularities are rational quotient singularities, and a klt surface is $\Q$-factorial \cite[Proposition~4.18 and Theorem~5.22]{KM}. The exceptional divisor over a point is an SNC tree of smooth rational curves of square at most $-2$. Its graph is a Hirzebruch--Jung chain in the cyclic case, and a three-armed quotient fork otherwise; see \cite[Sections~2--3]{Reid} and \cite[Theorem~8.30]{Riem}.
We write
\begin{equation}\label{eq:resolution}
 K_Y+\Delta=f^*K_X,\qquad
 \Delta=\sum_p\Delta_p=\sum_{p,i}(1-\alpha_{p,i})E_{p,i},
 \qquad P=-f^*K_X.
\end{equation}
At $p$ put
\[
 M_p=-\bigl(E_{p,i}E_{p,j}\bigr),\quad Q_p=M_p^{-1},
 \quad D_p=\det M_p,\quad b_{p,i}=-E_{p,i}^2.
\]
The matrices $M_p$ are positive definite with non-positive off-diagonal entries.

\begin{lemma}\label{lem:positive}
Every entry of $Q_p$ is positive. We have $\Delta\ge0$ and $\eps\le\alpha_{p,i}\le1$. The coefficients of $\Delta_p$ are all positive if $p$ is noncanonical, and all zero if $p$ is canonical. Moreover $D_p\Delta_p$ is integral.
\end{lemma}
\begin{proof}
Choose $c$ larger than all diagonal entries and the largest eigenvalue of $M_p$. Then $A=I-M_p/c$ has nonnegative entries and spectral radius less than one. Hence
\[
 M_p^{-1}=c^{-1}\sum_{j\ge0}A^j.
\]
The graph is connected, so paths show that some power has a positive entry in every prescribed position.
Writing $d_i=1-\alpha_{p,i}$, adjunction gives
\begin{equation}\label{eq:disc-matrix}
 M_pd=(b_{p,i}-2)_i.
\end{equation}
The right side is nonnegative. Positivity of the inverse proves the assertions about $\Delta_p$; the lower discrepancy bound is the $\eps$-lc hypothesis. Finally the adjugate matrix is integral, so $D_pd$ is integral.
\end{proof}

If $v_i$ is the valency of a vertex in an exceptional graph, the equivalent discrepancy equation is
\begin{equation}\label{eq:disc-recurrence}
 b_i\alpha_i=\sum_{j\sim i}\alpha_j+2-v_i.
\end{equation}
\begin{lemma}\label{lem:rational}
The surface $Y$ is rational. If $N$ is the number of exceptional components, then we have
\[
 \rho(Y)=\rho(X)+N,\qquad K_Y^2+\rho(Y)=10.
\]
\end{lemma}
\begin{proof}
By Kawamata--Viehweg vanishing, we have $H^i(X,\mathcal O_X)=0$ for $i>0$ \cite[Theorem~2.70]{KM}. Rational singularities give the same statement on $Y$. Since $K_YP=-P^2<0$, no positive multiple of $K_Y$ is effective. Castelnuovo's criterion therefore proves rationality.
For an integral divisor $L$ on $Y$, the divisor $f_*L$ is $\Q$-Cartier and $L-f^*f_*L$ is exceptional. Pullback classes and exceptional classes thus span $N^1(Y)_\R$. Negative definiteness proves their independence. Finally, $K^2+\rho=10$ on $\PP^2$ and on every Hirzebruch surface, and a point blowup preserves this sum. This proves the second equality.
\end{proof}

\subsection{Chains, forks, and Schur complements}
For a symmetric block matrix with $H$ invertible, the \emph{Schur complement} of $H$ is $b-u^TH^{-1}u$:
\begin{equation}\label{eq:schur}
 \det\begin{pmatrix}b&-u^T\\-u&H\end{pmatrix}
 =\det H\,(b-u^TH^{-1}u).
\end{equation}
Completing the square shows that, if $H$ is positive definite, the whole matrix is positive definite precisely when this scalar is positive. More generally, eliminating an invertible symmetric block preserves the inertia after adding the inertia of that block. Here \emph{inertia} means the numbers of positive, negative, and zero eigenvalues.

We write $[b_1,\ldots,b_s]$ for the chain with squares $-b_i$ and adjacent intersections one. An empty chain has determinant one. Expansion along an end gives the continuant recurrence
\[
 d_0=1,\quad d_1=b_1,\quad d_j=b_jd_{j-1}-d_{j-2}.
\]
When $b_i\ge2$, induction gives $d_j>d_{j-1}$ and $d_j\ge j+1$.

\begin{lemma}\label{lem:chain}
For a cyclic chain of determinant $r$, let $L_i,R_i$ be the determinants strictly to the left and right of vertex $i$. Then
\begin{equation}\label{eq:chain}
 (Q_p)_{ii}=\frac{L_iR_i}{r},\qquad
 \alpha_i=\frac{L_i+R_i}{r}.
\end{equation}
With $\alpha_0=\alpha_{s+1}=1$, we also have
\begin{equation}\label{eq:chain-mass}
 r=\sum_{i=0}^s\frac1{\alpha_i\alpha_{i+1}}.
\end{equation}
\end{lemma}
\begin{proof}
Deleting vertex $i$ separates the chain into two pieces, so its diagonal cofactor is $L_iR_i$. Expanding the end-column cofactors along the connecting path gives
$Q_{i1}=R_i/r$ and $Q_{is}=L_i/r$. The discrepancy equations are $M\alpha=e_1+e_s$, with $2e_1$ when $s=1$, proving \eqref{eq:chain}.
To prove \eqref{eq:chain-mass}, set $z_0=0$, $z_1=1$, and
$z_{i+1}=b_iz_i-z_{i-1}$. Then $z_{s+1}=r$. Since $\alpha$ satisfies the same recurrence,
\[
 z_{i+1}\alpha_i-z_i\alpha_{i+1}=1,
 \qquad
 \frac{z_{i+1}}{\alpha_{i+1}}-\frac{z_i}{\alpha_i}
 =\frac1{\alpha_i\alpha_{i+1}}.
\]
Summing up gives the identity.
\end{proof}

\begin{lemma}\label{lem:fork}
For a noncyclic quotient singularity,
\[
 D_p\le4/\eps,\qquad (Q_p)_{ii}\ge\alpha_i\quad\text{for every }i.
\]
If $\gamma_i=(Q_p)_{ii}-\alpha_i>0$, then $\gamma_i\ge1/D_p\ge\eps/4$.
\end{lemma}
\begin{proof}
Let $0$ be the central vertex. The three arm matrices have determinants $r_j\ge2$; let $q_j$ be the determinant left after deleting the attaching vertex of arm $j$. Put
$A=b_0-\sum_jq_j/r_j$. Eliminating the arms gives
\[
 D_p=r_1r_2r_3A,\qquad
 A\alpha_0=\frac1{r_1}+\frac1{r_2}+\frac1{r_3}-1>0.
\]
Indeed, the first discrepancy on arm $j$ is $(q_j\alpha_0+1)/r_j$.
Ordering the $r_j$ shows that their triple is
$(2,2,t)$, $(2,3,3)$, $(2,3,4)$, or $(2,3,5)$.
Thus $D_p\alpha_0$ is respectively $4,3,2$, or $1$. This proves the determinant bound.

Fix $i$ and put $u=M^{-1}e_i$. If $\ell_j$ is the outer tip of arm $j$, symmetry and \eqref{eq:disc-recurrence} give
\[
 \alpha_i=u_{\ell_1}+u_{\ell_2}+u_{\ell_3}-u_0.
\]
If $i=0$, the source-free arm equations give $u_{\ell_j}=u_0/r_j$, whence $\alpha_0\le u_0/2\le u_0$. If $i$ lies on arm $1$, the other two arms give
\[
 \alpha_i=u_{\ell_1}-u_0(1-1/r_2-1/r_3)\le u_{\ell_1}.
\]
If the tail beyond $i$ has determinant $T$, its source-free equations give $u_{\ell_1}=u_i/T\le u_i$. This proves the inverse estimate. Both $D_pQ_{ii}$ and $D_p\alpha_i$ are integral, so a positive gap is at least $1/D_p$.
\end{proof}

\begin{lemma}\label{lem:cyclic-gap}
At a cyclic vertex put $\gamma_i=Q_{ii}-\alpha_i$. At an endpoint, $\gamma_i=-1/r$. At an interior vertex,
\[
 \gamma_i=\frac{L_iR_i-L_i-R_i}{r}\ge0.
\]
Either $L_i=R_i=2$, in which case $\gamma_i=0$ and $r\le4/\eps$, or $\gamma_i\ge\eps/5$.
\end{lemma}
\begin{proof}
The formulas follow from \eqref{eq:chain}. For integers $L,R\ge2$ other than $(2,2)$,
\[
 \frac{LR-L-R}{L+R}\ge\frac15.
\]
If $L=2$ and $R\ge3$, the ratio is $(R-2)/(R+2)\ge1/5$. If $L,R\ge3$, it is at least $1/2$. Multiplication by $(L+R)/r=\alpha_i\ge\eps$ proves the assertion.
\end{proof}

\subsection{Complements and intersection budgets}
Recall that an \emph{$n$-complement} of $(X,0)$ is an effective boundary $B$ such that $(X,B)$ is lc and $n(K_X+B)\sim0$ as an integral Cartier divisor. In particular, $nB$ is integral. A Fano surface is \emph{exceptional} if $(X,G)$ is klt for every effective $G\sim_\R-K_X$.

\begin{proposition}\label{prop:complement}
An exceptional Fano surface has $V(X)\ge\cexc$. A non-exceptional klt Fano surface has an $n$-complement for some $n\in\{1,2,3,4,6\}$.
\end{proposition}
\begin{proof}
The second assertion is the surface complement theorem \cite[Inductive Theorem~2.3]{Shokurov}, in the form \cite[Lemma~3.1]{LiuComp}. Its $\R$-complementarity hypothesis holds by taking a general member of a sufficiently divisible anticanonical system and dividing by that multiple.
For an exceptional surface, \cite[Corollary~4.5]{LiuComp} gives an integer $I\le2\cdot84^{128\cdot42^5}$ with $IK_X$ Cartier and $K_X^2\ge1/I$. This proves the first assertion.
\end{proof}

Given a complement on a Fano surface, we use $P,\Delta,\Gamma,U$ as in \eqref{eq:intro-setup}. The pullback of an effective $\Q$-Cartier divisor is effective, by pulling back a local equation of a Cartier multiple. Thus
\begin{equation}\label{eq:Gamma}
 \Gamma=\Delta+U\ge0,\qquad 0\le\Gamma\le1,\qquad n\Gamma\ \text{is integral}.
\end{equation}
The integrality follows by pulling back a principal-divisor identity for $n(K_X+B)$.
We call a component of $\Gamma$ \emph{strict} if it is the strict transform of a component of $B$. If
$B=\sum b_iC_{X,i}$ and $t_i=(-K_X)C_{X,i}$, then
\begin{equation}\label{eq:degree-budget}
 t_i>0,\qquad b_i\ge1/n,\qquad V=\sum_i b_it_i,\qquad t_i\le nV.
\end{equation}

For an integral curve $C_X\subset X$ with strict transform $C\subset Y$, we define, at a singular point on $C_X$,
\begin{equation}\label{eq:local-budget}
 v_p=(CE_{p,i})_i,\quad
 \sigma_p=(1-\alpha_p)^Tv_p,\quad
 \tau_p=v_p^TQ_pv_p,\quad \lambda_p=\sigma_p+\tau_p.
\end{equation}
These vectors have nonnegative integral entries and are nonzero precisely at the singular points met by $C_X$.

\begin{lemma}\label{lem:budget}
If $C_X$ is a component of an $n$-complement, with coefficient $b$ and degree $t$, then
\begin{align}
 t&=2-2p_a(C)+C^2-\sum_p\sigma_p,\label{eq:adj-budget}\\
 C_X^2&=C^2+\sum_p\tau_p,\label{eq:pullback-budget}\\
 \sum_p\lambda_p&=2-2p_a(C)+C_X^2-t
                   \le2+n(n-1)V.\label{eq:global-budget}
\end{align}
Moreover $C_X^2\le t/b\le nt$. If $C^2>0$, then $V\ge1/n^2$.
If $J$ is the least common multiple of the determinants of the non-canonical points met by $C_X$, with $J=1$ for the empty set, then $Jt$ is a positive integer.
\end{lemma}
\begin{proof}
By adjunction, we have \eqref{eq:adj-budget}. The numerical pullback is
\[
 f^*C_X=C+\sum_{p,i}(Q_pv_p)_iE_{p,i}.
\]
Expanding its square gives \eqref{eq:pullback-budget}; adding proves the equality in \eqref{eq:global-budget}.
Distinct effective $\Q$-Cartier curves on a normal surface have nonnegative intersection: restrict a Cartier multiple of one to the normalization of the other. Hence
$t=BC_X\ge bC_X^2$. Since $V\ge bt$ and $b^{-1}\le n$, we obtain the inequality in \eqref{eq:global-budget}.
If $C^2>0$, write $U=bC+G$ with $G\ge0$ not containing $C$. Then
$t=UC\ge bC^2\ge b$, and $V\ge bt\ge1/n^2$.
Finally, the only nonintegral terms in \eqref{eq:adj-budget} are the intersections with $\Delta_p$. They vanish at canonical points, and $D_p\Delta_p$ is integral. This proves the last assertion.
\end{proof}

\begin{lemma}\label{lem:lambda}
At a cyclic point of determinant $r$, with $l=\sum_i(v_p)_i$, we have
$\lambda_p\ge l(1-1/r)$. At a non-cyclic point, $\lambda_p\ge l$.
In particular, each unit of non-canonical contact contributes at least $2/3$, and every canonical contact contributes at least $1/2$.
\end{lemma}
\begin{proof}
Let $d=1-\alpha$ and $Q=M^{-1}$. Since $d\ge0$, $Q$ is entry-wise positive, and $v_i^2\ge v_i$ for integral $v_i\ge0$,
\[
 d^Tv+v^TQv\ge\sum_i v_i(d_i+Q_{ii}).
\]
In the cyclic case,
$d_i+Q_{ii}=1-1/r+(L_i-1)(R_i-1)/r$.
In the fork case it is at least one by \cref{lem:fork}. A non-canonical cyclic point has $r\ge3$, while any cyclic point has $r\ge2$. These observations prove all assertions.
\end{proof}

\section{Weighted fibres and the cubic estimates}\label{sec:fibres}

We first prove a bound for an arbitrary smooth rational log Calabi--Yau pair. The ruling and discriminant construction is related to \cite[Section~2]{BL}; the estimate below allows arbitrary $\eps$ and boundary coefficients. Its essential ingredient is an exact change-of-weight formula.

\begin{proposition}\label{prop:potential}
Let $(S,\Theta)$ be a smooth projective rational $\eps$-lc log Calabi--Yau pair, where $0<\eps<1$. Unless $S=\PP^2$, choose a birational morphism to a Hirzebruch surface and the induced ruling. For every vertical prime curve set $a_T=1-\operatorname{coeff}_T\Theta$. Define
\[
 \Phi(S,\Theta)=\sum_{T\cap T'\text{ a node of a reduced fibre}}
                         \frac1{a_Ta_{T'}}.
\]
Then
\begin{equation}\label{eq:potential}
 \rho(S)-2\le\Phi(S,\Theta)\le4\eps^{-2}+12\eps^{-3}.
\end{equation}
\end{proposition}

\subsection{Blowups and horizontal multiplicity}
We can factor the chosen morphism as
\[
 S_0=\mathbb F_d\longleftarrow S_1\longleftarrow\cdots
       \longleftarrow S_r=S,\qquad r=\rho(S)-2.
\]
Such a factorization exists by the classification of smooth rational minimal surfaces and point-blowup factorization. If a contraction sequence ends at $\PP^2$ and $S\ne\PP^2$, stop at the preceding surface $\mathbb F_1$.
Let $\Theta_j$ be the push-forward boundary. Each pair is effective and crepant: the difference
$K_{S_{j+1}}+\Theta_{j+1}-\mu_j^*(K_{S_j}+\Theta_j)$ is exceptional and numerically trivial, hence zero. In particular, all intermediate pairs are $\eps$-lc.

We may write $\Theta_j=H_j+W_j$ with horizontal and vertical parts. Every reduced fibre is an SNC tree, as follows inductively from the two possible blowups of a reduced fibre. The weights of its components satisfy $\eps\le a_T\le1$ and are unchanged under strict transform. At a blowup centre set $h_j=\operatorname{mult}H_j$. If the centre is a smooth point of the reduced fibre, on a component of weight $a$, the new weight is
\begin{equation}\label{eq:smooth-blowup}
 c=1+a-h_j.
\end{equation}
If it is a node of weights $a,b$, the new weight is
\begin{equation}\label{eq:node-blowup}
 c=a+b-h_j.
\end{equation}
These follow by subtracting one from the total boundary multiplicity to obtain the new boundary coefficient. No SNC hypothesis on $H_j$ is needed.

\begin{lemma}\label{lem:horizontal}
For this sequence,
\begin{equation}\label{eq:horizontal}
 \sum_jh_j\le4+4/\eps,\qquad
 \sum_{j\text{ a node blowup}}h_j\le8.
\end{equation}
\end{lemma}
\begin{proof}
Let $F$ denote a general fibre on any model. We have $H_jF=2$. On $S_0$ write $W_0\equiv vF$, where $v\ge0$. Since $K_{S_0}^2=8$,
\[
 (-K_{S_0})H_0=8-2v\le8.
\]
Put $\lambda=2(1-\eps)/\eps$. We claim that $H_r+\lambda F$ is nef. Only a negative component $C$ of $H_r$ needs checking. If its coefficient is $b\le1-\eps$, then
\[
 2p_a(C)-2=(1-b)C^2-(\Theta_r-bC)C<0.
\]
Thus $C$ is rational and $(1-b)C^2\ge-2$. Consequently
$H_rC\ge bC^2\ge-2b/(1-b)\ge-\lambda$. Since $FC\ge1$, the claim follows. Therefore
\[
 (-K_S)H_r=\Theta_rH_r
 \ge-\lambda\Theta_rF=-4(1-\eps)/\eps.
\]
Under one blowup, $K'=\mu^*K+E$ and $H'=\mu^*H-hE$, so
$(-K')H'=(-K)H-h$. Telescoping gives the first inequality in \eqref{eq:horizontal}.

For the second, apply the canonical bundle formula to the ruling
$g:(S,\Theta)\to\PP^1$. It is a klt-trivial fibration: the generic fibre is $\PP^1$, its boundary is effective with coefficients less than one, and the rank condition is one. The boundary property and nefness of the moduli part give
\[
 K_S+\Theta\sim_\Q
 g^*(K_{\PP^1}+D+M),\qquad \deg M\ge0;
\]
see \cite[Theorem~2.7]{Ambro} and \cite[Definition~3.1 and Theorem~3.5]{FG}. On a smooth curve no further birational model of the base is necessary. If
$D=\sum_x(1-\tau_x)[x]$ and $\mu_T$ is the multiplicity of $T$ in the scheme-theoretic fibre over $x$, then
\[
 \tau_x=\lct(S,\Theta;g^*x)\le a_T/\mu_T\le1.
\]
Thus every discriminant coefficient is nonnegative and
\begin{equation}\label{eq:discriminant}
 \sum_x(1-\tau_x)\le2.
\end{equation}
A node blowup creates a persistent component of fibre multiplicity at least two and weight at most one. Its fibre therefore has discriminant coefficient at least $1/2$. At most four fibres undergo node blowups.

For one fixed fibre, let $R_j$ be its reduced divisor. A smooth-point blowup gives $R_{j+1}=\mu_j^*R_j$, while a node blowup gives $R_{j+1}=\mu_j^*R_j-E$. Hence
\[
 H_{j+1}R_{j+1}=
 \begin{cases}
 H_jR_j,&\text{at a smooth point},\\
 H_jR_j-h_j,&\text{at a node}.
 \end{cases}
\]
The initial intersection is two and the final one is nonnegative. The sum of node multiplicities in this fibre is at most two. Summing over at most four fibres proves the bound eight. This completes the proof.
\end{proof}

\subsection{The change in the edge sum}
\begin{lemma}\label{lem:two-variable}
If $a,c\in[\eta,1]$ and $h=1+a-c$, then $1/(ac)\le h/\eta^2$.
\end{lemma}
\begin{proof}
For fixed $a$, the function $c\mapsto ac(1+a-c)$ is concave. At the two endpoints it is $a^2\ge\eta^2$ and $a\eta(1+a-\eta)\ge\eta^2$. It is therefore at least $\eta^2$ throughout the interval.
\end{proof}

\begin{proof}[Proof of \cref{prop:potential}]
Initially the edge sum is zero. By \cref{lem:two-variable}, a smooth-point blowup adds an edge and increases it by
\[
 \frac1{ac}\le\frac{h_j}{\eps^2}.
\]
A node blowup replaces one edge by two. By \eqref{eq:node-blowup}, its exact increment is
\begin{equation}\label{eq:increment}
 \frac1{ac}+\frac1{bc}-\frac1{ab}
       =\frac{a+b-c}{abc}=\frac{h_j}{abc}\le\frac{h_j}{\eps^3}.
\end{equation}
In particular, a node blowup with $h_j=0$ leaves the sum unchanged. Using \cref{lem:horizontal}, we obtain
\[
 \Phi\le\eps^{-2}(4+4/\eps)+8\eps^{-3}.
\]
Every point blowup increases the number of reduced-fibre edges by one. There are $r$ final edges, each of weight at least one. Thus $r\le\Phi$, completing the proof.
\end{proof}

\subsection{The auxiliary boundary on a minimal resolution}
For a weak del Pezzo surface with $0<\eps<1$, rationality in \cref{lem:rational} allows us to choose a ruling on $Y\ne\PP^2$. Since $P$ is nef and big and $F$ is nonzero with $F^2=0$, Hodge index theorem gives $PF>0$. Choose a sufficiently divisible $q$ with $q^{-1}\le1-\eps$ and a general $G\in|-qK_X|$. We choose it to avoid the singular points and the curves contracted by the anticanonical morphism. Bertini theorem gives a smooth irreducible horizontal strict transform $G_Y$, disjoint from the exceptional locus. Then
\begin{equation}\label{eq:aux}
 \widehat\Gamma=\Delta+q^{-1}G_Y,\qquad
 K_Y+\widehat\Gamma\sim_\Q0
\end{equation}
is $\eps$-lc. Its vertical weights satisfy the useful compatibility
\begin{equation}\label{eq:weight-compat}
 a_T=
 \begin{cases}
 \alpha_T,&T\text{ is }f\text{-exceptional},\\
 1,&T\text{ is not }f\text{-exceptional}.
 \end{cases}
\end{equation}
By \cref{prop:potential}, this proves the Picard estimate in \cref{thm:cubic}. If $\eps=1$, then $\Delta=0$ and $V=K_Y^2$ is a positive integer, so $\rho(Y)\le9$. If $Y=\PP^2$, the assertion is immediate. The general pair statement also follows from \cref{prop:potential}; for $\eps=1$ an effective boundary must be zero, and a smooth rational surface cannot have $K_S\sim_\Q0$.

\subsection{From fibre edges to cyclic determinants}
Let $\Phi=\Phi(Y,\widehat\Gamma)$ and let $N$ be the total number of exceptional curves. We have
\begin{equation}\label{eq:N}
 N=\rho(Y)-\rho(X)\le1+\Phi,\qquad
 \sum_{E\text{ horizontal exceptional}}(1-\alpha_E)<2.
\end{equation}
The second inequality follows from $\Delta F=2-PF<2$ and $EF\ge1$ for a horizontal curve.

\begin{lemma}\label{lem:neighbor}
If adjacent exceptional components in a minimal resolution have discrepancies $a,b$, then $2a\le1+b$.
\end{lemma}
\begin{proof}
If the first curve has square $-d$ and valency $v$, then \eqref{eq:disc-recurrence} gives
$da=b+\sum_{\text{other neighbours}}\alpha_j+2-v\le b+1$.
Now $d\ge2$.
\end{proof}

\begin{proposition}\label{prop:mass}
For every $\eps$-lc weak del Pezzo surface,
$\Mcyc(X)\le344\eps^{-3}$. In particular, every noncanonical resolution determinant is at most $344\eps^{-3}$.
\end{proposition}
\begin{proof}
We may suppose $0<\eps<1$ and $Y\ne\PP^2$. Let $I$ be the sum of $1/(\alpha_E\alpha_{E'})$ over the internal edges of all noncanonical cyclic chains. We compare these edges with fibre edges.

Edges with both endpoints vertical contribute at most $\Phi$. Edges with a horizontal endpoint and both discrepancies at least $1/2$ have weight at most four; there are at most $N$ chain edges, so their contribution is at most $4N$.
For any remaining edge with a horizontal endpoint, some horizontal endpoint has discrepancy less than $3/4$, by \cref{lem:neighbor}. By \eqref{eq:N}, there are fewer than eight such horizontal curves. A chain vertex has degree at most two. This last class therefore contributes at most $16\eps^{-2}$. Thus
\begin{equation}\label{eq:I}
 I\le\Phi+4N+16\eps^{-2}.
\end{equation}

For a chain of length at least two, each endpoint term in \eqref{eq:chain-mass} is at most the weight of its adjacent internal edge. Consequently its determinant is at most three times its internal-edge sum.
An isolated exceptional curve has determinant $r=2/\alpha_E$. If it is vertical, its fibre is reducible and it meets a nonexceptional vertical curve: an adjacent exceptional curve would contradict isolation. By \eqref{eq:weight-compat}, that neighbour has weight one, so $r$ is twice the edge weight. Edges chosen for distinct isolated exceptional curves are distinct. These curves contribute at most $2\Phi$.
An isolated horizontal noncanonical exceptional curve has $1-\alpha_E\ge1/3$. There are fewer than six such curves by \eqref{eq:N}, each of determinant at most $2/\eps$. We conclude that
\begin{align*}
 \Mcyc
 &\le3I+2\Phi+12\eps^{-1}\\
 &\le17\Phi+12+48\eps^{-2}+12\eps^{-1}\\
 &\le204\eps^{-3}+116\eps^{-2}+12\eps^{-1}+12
 \le344\eps^{-3}.
\end{align*}
For a noncyclic point we use $D_p\le4/\eps$ from \cref{lem:fork}. The omitted cases have $\Mcyc=0$.
\end{proof}

\begin{remark}\label{rem:low-fibres}
The same budgets locate low-discrepancy curves without bounding their number in each fibre. For $0<\theta<1$, there are fewer than $2/(1-\theta)$ horizontal exceptional curves with $\alpha_E\le\theta$, and the vertical ones lie in at most $\lfloor2/(1-\theta)\rfloor$ fibres. The first assertion follows from \eqref{eq:N}; the second follows from \eqref{eq:discriminant}, since such a fibre has discriminant coefficient at least $1-\theta$.
\end{remark}

\section{Local ages and negative complement curves}\label{sec:curves}

We now assume that $X$ is Fano and has a fixed $n$-complement with $n\le6$. Our aim is to prove a uniform contact description and negativity when $V$ has quadratic order. The determinant bound of \cref{prop:mass} is not used in this section.

\subsection{The positive endpoint excess}
For a cyclic quotient $\frac1r(1,q)$, let $\overline{jq}$ denote the least non-negative residue modulo $r$. We use the cyclic age formula
\begin{equation}\label{eq:age}
 \mld\!\left(\tfrac1r(1,q)\right)
       =\min_{1\le j<r}\frac{j+\overline{jq}}r,
\end{equation}
which follows from the toric discrepancy formula \cite[Section~1]{AmbroMld}, or directly from the Hirzebruch--Jung resolution.

\begin{lemma}\label{lem:endpoint-excess}
If the sum of the two endpoint discrepancies of a cyclic chain is greater than one, its excess is at least $\mld_p(X)^2/2$.
\end{lemma}
\begin{proof}
Let $q^\vee\in\{1,\ldots,r-1\}$ be the inverse of $q$ modulo $r$. The endpoint discrepancies are $(q+1)/r$ and $(q^\vee+1)/r$. Put
\[
 v=q+q^\vee+2-r>0,\qquad \mu=\mld_p(X).
\]
We must show $\mu^2\le2v/r$. The identity $qq^\vee\equiv1\pmod r$ gives
\[
 q^2+(2-v)q+1\equiv0\pmod r.
\]
The lattice
$\Lambda=\{(x,y)\in\Z^2:y\equiv qx\pmod r\}$ has covolume $r$ and is preserved by
$T_v(x,y)=(-y,x+(2-v)y)$.

Suppose first that $v\ge2$. If $2(v-1)\ge r$, the estimate follows from $\mu\le1$. Otherwise choose
$\sqrt{2(v-1)r}<B<r$. The centrally symmetric diamond
\[
 |x|+(v-1)|y|<B
\]
has area $2B^2/(v-1)>4r$. Minkowski's convex-body theorem \cite[Chapter~III]{Cassels} gives a nonzero lattice vector in it. Neither coordinate is zero, since a nonzero axis vector in $\Lambda$ has length at least $r$. If the coordinates have the same sign, change the overall sign to make them positive. Otherwise take $x>0>y$ and apply $T_v$. Both new coordinates are positive and their sum is $x+(v-1)|y|<B$.
In either case \eqref{eq:age} gives $\mu<B/r$. Letting $B$ decrease, we obtain
$\mu^2\le2(v-1)/r<2v/r$.

If $v=1$, use $Q(x,y)=x^2+xy+y^2$. Its values on $\Lambda$ are divisible by $r$. The ellipse $Q<2r$ has area $4\pi r/\sqrt3>4r$, so it contains a nonzero lattice vector, necessarily of norm $Q=r$. The map $T_1(x,y)=(-y,x+y)$ preserves both the lattice and the form; its six powers carry a vector into the cone $x,y\ge0$. No coordinate can be zero, since then the other is divisible by $r$ and $Q=r$ is impossible for $r\ge2$. Finally
\[
 \tfrac34(x+y)^2\le Q(x,y)=r.
\]
If $x+y<r$, the age formula gives $\mu^2\le4/(3r)$; otherwise the same bound follows from $\mu\le1$. In particular $\mu^2\le2/r$, as required.
\end{proof}

\subsection{The contact configuration}
\begin{theorem}\label{thm:contact}
Suppose $V<\eps^2/1000$. Every strict complement component is a smooth rational $(-1)$-curve on $Y$. It meets exactly two exceptional fibres, both non-canonical, once and transversely at a smooth point of one exceptional component in each fibre. At least one contact is a cyclic endpoint.
\end{theorem}
\begin{proof}
Fix a component $C_X$ and write $C$ for its strict transform. By \cref{lem:budget},
\begin{equation}\label{eq:contact-budget}
 t\le6V,\qquad C_X^2\le6t,\qquad
 \sum_p\lambda_p\le2+30V<2.03.
\end{equation}
If $C^2>0$, then $V\ge1/36$, a contradiction. Adjunction therefore gives $p_a(C)=0$ and $C^2\in\{-1,0\}$.

Let $l$ be the total intersection multiplicity with non-canonical exceptional curves. By \cref{lem:lambda}, $2l/3<2.03$, so $l\le3$. If $l=0$, then $t$ is a positive integer. If $l=1$, adjunction gives $t\ge\eps$. Both contradict $t\le6V$.
If $l=3$ and any contacted cyclic determinant is at least four, the local contributions sum to at least
$3/4+2/3+2/3=25/12>2.03$. A fork gives an even larger sum. Thus all three units would lie at cyclic points of order three, with chain $[3]$ and discrepancy $2/3$. Adjunction would give $t=0$ or $1$, again impossible. Hence $l=2$. If $C^2=0$, then $t\ge2\eps$, so $C^2=-1$.

We next exclude the canonical contacts case. Such a contact contributes at least $1/2$. If the two non-canonical units lie at one cyclic point of order $r$, then
$2(1-1/r)+1/2<2.03$, so $r\le4$. The denominator assertion of \cref{lem:budget} gives $t\ge1/r$, impossible. A fork contributes at least two before the canonical contribution and is also impossible.
If the two units lie at distinct points and one is a fork, the total is at least $1+2/3+1/2>2.03$. Otherwise their orders satisfy
$1/r_1+1/r_2>0.47$. Ordering $3\le r_1\le r_2$, we obtain either $r_1=3$, $r_2\le7$, or $r_1=r_2=4$. Thus $t\ge1/(r_1r_2)\ge1/49$, again impossible.

We must also exclude two units at the same non-canonical point. Put $\gamma_i=Q_{ii}-\alpha_i$. For a contact vector $2e_i$,
\[
 t=2\alpha_i-1,\qquad
 \lambda=2+2\gamma_i+2Q_{ii}.
\]
At an interior cyclic or fork vertex this is at least $2+2\eps$, contrary to \eqref{eq:contact-budget}. At an endpoint it is $3+t-4/r$, which forces $r\le4$ and then $t\ge1/4$.
For a contact vector $e_i+e_j$, $i\ne j$, we have
\[
 t=\alpha_i+\alpha_j-1,\qquad
 \lambda=2+\gamma_i+\gamma_j+2Q_{ij}.
\]
For a fork, every inverse entry is a positive integral multiple of $1/D_p$, so $\lambda\ge2+\eps/2$. For a cyclic point with a positive interior gap, $\lambda\ge2+\eps/5$: a possible endpoint term $-1/r$ is offset by $2Q_{ij}\ge2/r$. If an interior gap is zero, then $r\le4/\eps$, so $t\ge1/r\ge\eps/4$. These are contradictions. The remaining possibility is the two endpoints of one cyclic chain. In this case \cref{lem:endpoint-excess} gives $t\ge\eps^2/2$, again impossible.

There are therefore two distinct non-canonical points with unit contact vectors. Adjunction and pullback give
\begin{equation}\label{eq:two-contact}
 t=\alpha_1+\alpha_2-1,\qquad
 C_X^2=t+\gamma_1+\gamma_2,\qquad
 \gamma_1+\gamma_2\le5t.
\end{equation}
If neither contact is a cyclic endpoint, the two gaps are nonnegative. Any positive gap is at least $\eps/5$, contradicting $5t\le30V$. If both vanish, their determinants are at most $4/\eps$, and $t\ge\eps^2/16$, also impossible.
Finally, arithmetic genus zero makes $C$ smooth rational. A unit contact vector gives a single transverse intersection and excludes an exceptional node.
\end{proof}

\subsection{Two age estimates}
\begin{lemma}\label{lem:two-germ}
For two cyclic quotient germs, let $\alpha_i$ be the discrepancy at a chosen endpoint and $\tau_i$ the diagonal inverse entry there. Suppose both germs are $\eps$-lc. If
$t=\alpha_1+\alpha_2-1>0$ and $\tau_1+\tau_2\ge1$, then $t\ge\eps^2/4$.
\end{lemma}
\begin{proof}
We may orient the germs as $\frac1{r_i}(1,q_i)$ so that $\tau_i=q_i/r_i$ and $\alpha_i=(q_i+1)/r_i$. Set $x_i=q_i/r_i$ and $\delta=x_1+x_2-1\ge0$. Then
$t=\delta+1/r_1+1/r_2$.
For $t\ge1/4$ there is nothing to prove. Otherwise let $Q_0=\lfloor t^{-1/2}\rfloor\ge2$. Among the $Q_0+1$ fractional parts $0,x_1,\ldots,Q_0x_1$ on the circle, one gap is at most $1/(Q_0+1)$. Thus some $1\le j\le Q_0$ satisfies
$\|jx_1\|\le1/(Q_0+1)$.
Since $r_i>1/t>Q_0$, this $j$ is admissible in both age formulas.

If $\{jx_1\}\le1/(Q_0+1)$, then
\[
 \mld_1\le j/r_1+\{jx_1\}
 \le Q_0t+1/(Q_0+1)<2\sqrt t.
\]
Otherwise we write $\{jx_1\}=1-d$ with $d\le1/(Q_0+1)$. Since $x_2=1-x_1+\delta$, we have $\{jx_2\}\le d+j\delta$, and therefore
\[
 \mld_2\le j/r_2+\{jx_2\}
 \le1/(Q_0+1)+Q_0t<2\sqrt t.
\]
In both cases $\eps<2\sqrt t$, proving the claimed weak inequality.
\end{proof}

\begin{lemma}\label{lem:cancellation}
Let $D,r,a,k$ be positive integers with $D\le A/\eps$, where $A\ge1$. Suppose $\frac1r(1,r-k-1)$ is an $\eps$-lc cyclic quotient and
\[
 0<t=\frac aD-\frac kr<\frac{\eps^2}{A^2}.
\]
Then $t<1/r$.
\end{lemma}
\begin{proof}
If $r\le D$, integrality gives $t\ge1/(Dr)\ge1/D^2$, a contradiction. Hence $D<r$. Put $I=ar-kD=Drt\in\Z_{>0}$. Since $Dt<\eps/A\le1$, we have $I<r$. If $t\ge1/r$, then $I\ge D$ and the residue of $D(r-k-1)$ modulo $r$ is $I-D\in[0,r)$. It is nonzero by coprimality. The age formula at $j=D$ gives
\[
 \eps\le\mld\le\frac{D+(I-D)}r=Dt,
\]
so $t\ge\eps/D\ge\eps^2/A$, contradicting the assumption.
\end{proof}

\subsection{Contractions and mass monotonicity}
\begin{lemma}\label{lem:contract}
Let $C_X$ be an integral curve on an $\eps$-lc Fano surface, with $C_X^2=-w<0$ and $t=(-K_X)C_X>0$. There is a projective birational morphism $g:X\to X'$ contracting exactly $C_X$. The target is $\eps$-lc Fano, and
\begin{equation}\label{eq:contract}
 K_X=g^*K_{X'}+(t/w)C_X,\quad
 V'=V+t^2/w,\quad
 V^{-1}-(V')^{-1}=\frac{t^2}{wVV'}.
\end{equation}
For a complement $B$ containing $C_X$, the push-forward $B'=g_*B$ is an lc $n$-complement and
$K_X+B=g^*(K_{X'}+B')$. Exactly one boundary component disappears and the other coefficients are unchanged.
\end{lemma}
\begin{proof}
Choose an ample Cartier divisor $A_X$ and put
$L=A_X+(A_XC_X/w)C_X$. Then $LC_X=0$, $LD>0$ for every other integral curve $D$, and
$L^2=A_X^2+(A_XC_X)^2/w>0$. A Cartier multiple of $L$ is semi-ample by \cite[Theorem~3.3]{KM}, since its positive multiples minus $K_X$ are nef and big. Its morphism contracts precisely $C_X$.
Equivalently, the cone theorem identifies its null face with the extremal ray generated by $C_X$; the surface contraction theorem gives a $\Q$-factorial target \cite[Theorem~3.7]{KM}, \cite[Theorem~3.3]{Fujino}.

Write $K_X=g^*K_{X'}+aC_X$. Intersecting with $C_X$ gives $a=t/w>0$. On a common resolution,
\[
 a(E,X',0)=a(E,X,0)+(t/w)\ord_E(C_X)\ge\eps.
\]
Effectivity of the pullback of $C_X$ justifies the inequality. Squaring gives the volume formula. The pullback $-g^*K_{X'}=-K_X+(t/w)C_X$ has positive intersection with every noncontracted curve. Together with $V'>0$, the surface Nakai--Moishezon criterion proves ampleness.

The difference $K_X+B-g^*(K_{X'}+B')$ is exceptional and numerically trivial. Intersection with $C_X$ makes it zero. Thus lc descends by crepancy. Pushing forward the principal-divisor identity for $n(K_X+B)$ proves $n(K_{X'}+B')\sim0$. Distinct surviving prime divisors have distinct images, proving the support and coefficient assertions.
\end{proof}

\begin{lemma}\label{lem:det-contract}
In the configuration of \cref{thm:contact}, if $w>0$, the new resolution determinant is
\[
 D'=D_1D_2w,
\]
where $D_1,D_2$ are the two old determinants; for a smooth image point we set $D'=1$.
\end{lemma}
\begin{proof}
The negative intersection matrix of the exceptional divisor over the new point on $Y$ has blocks $M_1,M_2$, and the entry one for $C$. Eliminating the old blocks leaves
$1-\tau_1-\tau_2=-C_X^2=w$. This proves the formula on $Y$.
Passing to the minimal resolution contracts exceptional $(-1)$-curves. Each contraction takes the Schur complement of a diagonal entry one and hence preserves the determinant. The empty final matrix has determinant one. This argument does not assume that the cyclic or noncyclic type is preserved.
\end{proof}

\begin{theorem}\label{thm:negative}
Suppose $V<\eps^2/10000$. Every complement component has negative square on $X$ and contracts to an $\eps$-lc Fano surface. Under any such contraction $\Mcyc$ does not increase.
More precisely, for each component there are $R_i>0$ and $c_i\in\{1,2\}$ such that
\begin{equation}\label{eq:wR}
 w_i=-C_{X,i}^2=R_i^{-1}-c_it_i.
\end{equation}
We can assign this component to an adjacent cyclic germ with determinant at least $R_i$.
\end{theorem}
\begin{proof}
By \cref{thm:contact}, there are two unit contacts, at least one a cyclic endpoint. We have $t\le6V<0.0006\eps^2$.
For two endpoints of orders $r,s$, \eqref{eq:two-contact} gives
$w=1/r+1/s-t$. If $w\le0$, then $\tau_1+\tau_2\ge1$, and \cref{lem:two-germ} gives $t\ge\eps^2/4$, impossible. Here $R=rs/(r+s)$ and $c=1$.

Suppose that the other contact is a fork, of determinant $D\le4/\eps$. If its gap $\gamma$ is positive, then $\gamma\ge\eps/4$ and $\gamma-1/r\le5t$ imply $r\le8/\eps$. Thus $t\ge1/(Dr)\ge\eps^2/32$, a contradiction. If the gap is zero, its discrepancy is $a/D$ and $w=1/r-t$. Apply \cref{lem:cancellation} with $A=4$ to obtain $w>0$. Here $R=r$ and $c=1$.

It remains to consider a cyclic interior contact. Let its determinant be $d$ and its side determinants be $L,R\ge2$. Put
\[
 S=L+R,\qquad H=LR-L-R,\qquad
 \alpha_{\rm end}=1-k/r,\qquad \Theta=S-kH.
\]
Then $k\ge1$ and
\begin{equation}\label{eq:mixed}
 t=S/d-k/r,\qquad
 C_X^2=\frac{k+1}{k}t-\frac{\Theta}{kd}
       =\frac{LR}{S}t-\frac{\Theta}{Sr}.
\end{equation}
If $H=0$, then $L=R=2$, $d\le4/\eps$, and \cref{lem:cancellation} again gives $w=1/r-t>0$.
Assume $H>0$. By \cref{lem:cyclic-gap},
\begin{equation}\label{eq:rsmall}
 1/r\ge\eps/5-5t>\eps/6,\qquad r<6/\eps.
\end{equation}

If $\Theta<0$, then $1/L+1/R+1/(k+1)<1$. For three integers at least two, the largest reciprocal sum strictly below one is $41/42$. Indeed, after ordering $a\le b\le c$, one checks $a\ge4$, $a=3$, and $a=2$; the last case has maximal triple $(2,3,7)$. Since $(k+1)LR/S\ge2$, we obtain $-\Theta/S\ge1/21$. Equations \eqref{eq:mixed} and \eqref{eq:rsmall} now give
$C_X^2>1/(21r)>\eps/126$, contradicting $C_X^2\le36V$.
If $\Theta=0$, the reciprocal triples are $(2,3,6)$, $(2,4,4)$, and $(3,3,3)$. Hence $S\le9$, $d\le9/\eps$, and $t\ge1/(dr)>\eps^2/54$, again impossible.

If $\Theta>0$, the reciprocal triples are $(2,2,h)$ or $(2,3,3)$, $(2,3,4)$, $(2,3,5)$. Apart from $L=R=2$, already treated, either $S\le8$ or, after exchange, $L=2$ and $k=1$. The bounded case gives $d\le8/\eps$ and $t>\eps^2/48$.
We are left with $L=2$, $k=1$. Put
\[
 I=r(R+2)-d=drt\in\Z_{>0}.
\]
Direct substitution gives
\begin{equation}\label{eq:last-mixed}
 w=\frac4d-2t=\frac{4r-2I}{dr}.
\end{equation}
If $d\le2r$, then $t\ge1/(2r^2)>\eps^2/72$. Thus $2r<d$. Orient the chain with its determinant-two side first. Its word begins $[2,b,\ldots]$, and the quotient parameter $q$ satisfies $2q=d+R$ by the continuant recurrence.
If $w\le0$, then $I\ge2r$. Also $I/d=rt<6t/\eps<1$. At $j=2r<d$, the residue of $jq$ is $I-2r\in[0,d)$, and coprimality excludes zero. The age formula gives
\[
 \eps\le\mld\le\frac{2r+(I-2r)}d=rt<6t/\eps<\eps,
\]
a contradiction. This proves negativity. In this last case $R_i=d/4$, $c_i=2$; all other mixed cases have $R_i=r$, $c_i=1$. The stated assignment is now immediate.

The contractions exist by \cref{lem:contract}. We shall verify mass monotonicity using \cref{lem:det-contract}. For two endpoints,
$D'=r+s-rst<r+s$. For a cyclic interior and an endpoint,
\[
 D'=(k+1)d-LRr,\qquad
 d+r-D'=(L-1)(R-1)r+drt>0.
\]
If the new point is not noncanonical cyclic, its contribution to $\Mcyc$ is zero, so these inequalities suffice.
For a fork only the zero-gap case remains, and $D'=D(1-rt)<D$. Since $t\ge1/(Dr)$,
\[
 D/r\le D^2t\le(16/\eps^2)6V<1.
\]
Thus $D'<D<r$. Even if the new point is cyclic, its determinant is less than the old cyclic contribution. All other singular points are unchanged.
\end{proof}

\begin{corollary}\label{cor:snc}
Under $V<\eps^2/10000$, distinct strict complement components are disjoint on $Y$. Consequently $(Y,\Gamma)$ is an SNC, hence dlt, pair.
\end{corollary}
\begin{proof}
For two endpoints we have $w_i<2/3$, since both cyclic orders are at least three. For a mixed contact, $w_i=1/r-\gamma-t_i<1/3$. Thus always $w_i<2/3$.
If two strict transforms meet, their intersection on $X$ is $h=C_{X,i}C_{X,j}\ge1$, since numerical pullback adds nonnegative corrections. Writing their coefficients as $b_i,b_j$, we have
\[
 t_i\ge b_jh-b_iw_i,\qquad
 t_j\ge b_ih-b_jw_j.
\]
Hence $t_i+t_j>(b_i+b_j)/3\ge2/(3n)$, and
$V\ge(t_i+t_j)/n>2/(3n^2)\ge1/54$, a contradiction.
The remaining components are exceptional SNC curves. By \cref{thm:contact}, every strict component meets them transversely away from their nodes. This proves the assertion.
\end{proof}

\section{Low discrepancies and summable contraction costs}\label{sec:low}

We retain a fixed bounded complement and choose a ruling on its minimal resolution. The fixed boundary $\Gamma$ controls coefficients and adjunction. The auxiliary boundary $\widehat\Gamma$ from \eqref{eq:aux} controls fibre weights and the canonical bundle formula. In particular, generic klt is required only for $\widehat\Gamma$, not for $\Gamma$.
For the fibre estimates we assume $0<\eps<1$ and $Y\ne\PP^2$. If $\eps=1$ or $Y=\PP^2$, then $V\ge1$ by \cref{lem:rational}, so all the volume lower bounds in this section are automatic.

Put $\delta_0=1/6$, and let $L$ be the number of vertical exceptional curves with $\alpha_E<\delta_0$. Since $\Gamma\ge\Delta$ and $n\Gamma$ is integral,
\begin{equation}\label{eq:low-floor}
 \alpha_E<1/n\quad\Longrightarrow\quad
 \operatorname{coeff}_E\Gamma=1.
\end{equation}
Thus every curve counted by $L$ is a floor component.

\subsection{A refined fibre estimate}
\begin{lemma}\label{lem:low-valency}
If $E$ is counted by $L$, its valency in the final reduced-fibre graph satisfies
$d_F(E)\le2n+2/\alpha_E$.
\end{lemma}
\begin{proof}
Applying adjunction to the rational coefficient-one component $E$ gives $(\Gamma-E)E=2$. Every vertical neighbour in $\Supp\Gamma$ has coefficient at least $1/n$, so there are at most $2n$ such neighbours.
If a neighbour $T$ is outside $\Supp\Gamma$, then its coefficient in $U$ is zero because $\Gamma=\Delta+U$ with both summands effective. The coefficient of $E$ in $U$ is $\alpha_E$. Consequently
$PT=UT\ge\alpha_E(ET)\ge\alpha_E$.
Write the complete fibre as $\sum_T\mu_TT$, with $\mu_T\ge1$. Nefness yields
\[
 \sum_{\text{these neighbours }T}PT
 \le\sum_T\mu_TPT=PF=2-\Delta F\le2.
\]
There are at most $2/\alpha_E$ neighbours of the second kind.
\end{proof}

\begin{proposition}\label{prop:low-potential}
For the auxiliary weights,
\begin{equation}\label{eq:low-potential}
 \Phi\le(86L+864)\eps^{-2},\qquad
 \Mcyc(X)\le(430L+4404)\eps^{-2}.
\end{equation}
\end{proposition}
\begin{proof}
Call a weight \emph{low} if it is less than $\delta_0$. Every low component on an intermediate model persists to a low vertical exceptional curve on $Y$, by \eqref{eq:weight-compat}. Thus at most $L$ blowups create a low component.

For a smooth-point blowup with old weight $a$ and new weight $c$, we distinguish three cases. If $c<\delta_0$, its increment is at most $\eps^{-2}$. If $a<\delta_0\le c$, it is at most $(\delta_0\eps)^{-1}$. The number of the latter events on a fixed persistent low curve is at most its final fibre valency: a smooth-point blowup raises that valency by one, and a node blowup preserves the valency of each old endpoint. By \cref{lem:low-valency}, their total contribution is at most
\[
 \frac{L(2n+2/\eps)}{\delta_0\eps}.
\]
If $a,c\ge\delta_0$, \cref{lem:two-variable} bounds the increment by $h/\delta_0^2$.

For a node blowup creating a low component, the increment in \eqref{eq:increment} is at most $1/(ac)+1/(bc)\le2\eps^{-2}$. Combining this with the first smooth-point case, all events creating a low component contribute at most $2L\eps^{-2}$.
For a node blowup with $c\ge\delta_0$, the equality $c=a+b-h$ implies $\max(a,b)\ge\delta_0/2$. Therefore
\[
 \frac h{abc}\le\frac{2h}{\eps\delta_0^2}.
\]
This includes the possibility that both old weights are low. By \cref{lem:horizontal},
\begin{align*}
 \Phi
 &\le2L\eps^{-2}+
 \frac{L(2n+2/\eps)}{\delta_0\eps}
 +\delta_0^{-2}(4+4/\eps)+\frac{16}{\eps\delta_0^2}\\
 &=14L\eps^{-2}+12nL\eps^{-1}+144+720\eps^{-1}\\
 &\le(86L+864)\eps^{-2}.
\end{align*}

For the mass bound, let $I$ again be the internal-edge sum for non-canonical cyclic chains. Every vertex in such a chain has positive $\Delta$-coefficient, so every horizontal vertex belongs to $\Supp\Gamma$. Since $\Gamma F=2$ and each nonzero coefficient is at least $1/n$, there are at most $2n$ such horizontal vertices. They meet at most $4n$ internal chain edges. All other internal edges are vertical fibre edges. Hence
$I\le\Phi+4n\eps^{-2}$.
The endpoint and isolated-curve estimates in the proof of \cref{prop:mass} now give
\[
 \Mcyc\le3I+2\Phi+12\eps^{-1}
 \le5\Phi+12n\eps^{-2}+12\eps^{-1}
 \le(430L+4404)\eps^{-2}.
\]
\end{proof}

\subsection{Choosing a contraction}
\begin{lemma}\label{lem:average}
Suppose $V<\eps^2/10000$. Write $B=\sum_{i=1}^m b_iC_{X,i}$, and use $w_i,R_i,t_i$ from \cref{thm:negative}. Put
\[
 W=\sum_i b_i,\qquad S_R=\sum_i b_iR_i,\qquad
 A_B=\sum_i b_i\sqrt{w_i},\qquad M=\Mcyc(X).
\]
Then $M>0$, $W\ge m/n$, $S_R\le2M$, and
\begin{equation}\label{eq:average}
 A_B\ge \frac{W^{3/2}}{\sqrt{2M}}-\sqrt{2WV}.
\end{equation}
If $W\ge4\sqrt{MV}$, some complement component has reciprocal-volume contraction cost at most
\begin{equation}\label{eq:cost}
 \frac{8M}{W^3}\le\frac{1728M}{m^3}.
\end{equation}
\end{lemma}
\begin{proof}
Every strict component has a cyclic contact, so $M>0$. For one cyclic germ, we may write
$c_j=1-\operatorname{coeff}_{E_j}\Gamma$, set $c_0=c_{s+1}=1$, and put
$m_j=\sum_i b_i(C_iE_j)$. Intersecting $K_Y+\Gamma$ with $E_j$ gives
\[
 c_{j-1}+c_{j+1}-2c_j=(q_j-2)c_j+m_j,
 \qquad q_j=-E_j^2.
\]
Summing up yields
\begin{equation}\label{eq:contact-mass-budget}
 \sum_jm_j=2-c_1-c_s-\sum_j(q_j-2)c_j\le2.
\end{equation}
Assign each $R_i$ to an adjacent cyclic germ as in \cref{thm:negative}. All contacts are simple. Thus the sum of the assigned $b_i$ at each germ is at most two, and $S_R\le2M$.

From \eqref{eq:wR}, $\sqrt{w_i}\ge R_i^{-1/2}-\sqrt{2t_i}$. Jensen's inequality for $x^{-1/2}$ and Cauchy--Schwarz give
\[
 \sum_i b_iR_i^{-1/2}\ge W^{3/2}/\sqrt{S_R},
 \qquad
 \sum_i b_i\sqrt{2t_i}\le\sqrt{2W\sum_i b_it_i}=\sqrt{2WV}.
\]
This proves \eqref{eq:average}.
Let $s=\min_i t_i/\sqrt{w_i}$. Then $V=\sum_i b_it_i\ge sA_B$. For a minimizing component, \eqref{eq:contract} gives cost
\[
 \frac{s^2}{V(V+s^2)}
 \le\frac1{A_B^2+V}\le A_B^{-2}.
\]
The condition $W\ge4\sqrt{MV}$ makes the error term in \eqref{eq:average} at most half its first term. Hence $A_B^2\ge W^3/(8M)$. Since $n\le6$, \eqref{eq:cost} follows.
\end{proof}

\begin{proposition}\label{prop:initial-cap}
Suppose $X$ has a fixed bounded complement and $\Mcyc(X)\le C\eps^{-2}$ for some $C>0$. Then
\begin{equation}\label{eq:cap-volume}
 V(X)\ge
 \frac{\eps^2}{\max\{10000,144C\}+432C}.
\end{equation}
If $C\ge625/9$, the denominator can be replaced by $576C$.
\end{proposition}
\begin{proof}
Put $h=\min\{1/10000,1/(144C)\}$. If $V\ge h\eps^2$, the claimed bound is weaker. Otherwise contract a component chosen by \cref{lem:average} and retain the push-forward complement. Repeat while $V<h\eps^2$. By \cref{thm:negative}, the mass cap persists.

There are always at least two complement components below this threshold: if $B=bC_X$, then $C_X^2=V/b^2>0$, contrary to \cref{thm:negative}. Thus
\[
 W\ge2/n\ge1/3,\qquad
 4\sqrt{MV}<4\sqrt{Ch}\le1/3.
\]
The cost at a stage with $m$ components is at most $1728C\eps^{-2}/m^3$.
The number of components decreases by exactly one at each step, so the process must reach a surface of volume at least $h\eps^2$. Each integer $m\ge2$ occurs at most once. Since
\[
 \sum_{m=2}^\infty m^{-3}
 \le\frac18+\int_2^\infty x^{-3}\,dx=\frac14,
\]
telescoping the reciprocal volumes gives
$V(X)^{-1}\le h^{-1}\eps^{-2}+432C\eps^{-2}$.
If $C\ge625/9$, then $144C\ge10000$, proving the last assertion.
\end{proof}

\begin{corollary}\label{cor:low-volume}
If one ruling on the initial minimal resolution has $L\le L_0$, then
\[
 V(X)\ge\frac{\eps^2}{576(430L_0+4404)}.
\]
\end{corollary}
\begin{proof}
This is an immediate Corollary of \cref{prop:low-potential,prop:initial-cap}. The count $L$ is needed only on the initial surface; mass monotonicity transports its consequence to all subsequent contractions.
\end{proof}

\begin{remark}\label{rem:pure-cubic}
The same argument immediately gives a pure cubic estimate. Apply \cref{prop:initial-cap} with $C=344/\eps$, using \cref{prop:mass}. Since $\eps\le1$,
\[
 V(X)\ge\eps^3/198144
\]
for a non-exceptional Fano surface. Together with \cref{lem:ac,prop:complement}, this gives $\min\{\cexc,\eps^3/198144\}$ for every weak del Pezzo surface. No bound on the initial number of complement components is required.
\end{remark}

\section{The global boundary and a half-integral complement}\label{sec:boundary}

We first explain why the boundary floor must have a one-dimensional connected graph when the volume is sufficiently small. This is where the local contraction estimates are combined with the global geometry of a log Calabi--Yau pair.

\subsection{The connected floor}
We use the following consequence of \cite[Proposition~5.1]{KK}: for a projective dlt log Calabi--Yau pair, its non-klt locus is connected, or it has exactly two connected components and the pair is plt. This is the case where the base in that proposition is a point.

\begin{proposition}\label{prop:floor}
Suppose $V<\hzero\eps^2$. For any bounded complement, the support of $\lfloor\Gamma\rfloor$ is a nonempty connected chain or cycle of smooth rational curves, and it contains an intersection of floor components.
\end{proposition}
\begin{proof}
The volume assumption implies that $(Y,\Gamma)$ is SNC by \cref{cor:snc}.
If its floor components were pairwise disjoint, each would be a connected component of the non-klt locus. The theorem just quoted gives at most two such components, including the case of an empty floor. By \eqref{eq:low-floor}, $L\le2$ for any ruling. Then \cref{cor:low-volume} gives
\[
 V\ge\frac{\eps^2}{576(430\cdot2+4404)}
   =\hzero\eps^2,
\]
a contradiction. Thus the floor has an edge.
If its support were disconnected, the same theorem would imply that $(Y,\Gamma)$ is plt. However, an intersection of two coefficient-one divisors on a smooth surface is not plt: blowing up their intersection point gives a log-discrepancy-zero exceptional divisor. Thus the floor is connected.

All its components are smooth rational curves: exceptional ones by the quotient resolution, and strict ones by \cref{thm:contact}. For a floor component $E$, adjunction gives $(\Gamma-E)E=2$. Every intersection with another floor component uses one unit of this intersection. Hence the floor graph has degree at most two at every vertex. A finite connected graph of maximum degree two with an edge is either a chain or a cycle.
\end{proof}

\begin{lemma}\label{lem:cycle-floor}
In \cref{prop:floor}, put $S=\lfloor\Gamma\rfloor$ and $H'=\Gamma-S$. If the floor is a cycle, then $H'=0$ and $S\sim-K_Y$. If it is a chain $E_1,\ldots,E_\ell$, then
\begin{equation}\label{eq:end-budget}
 H'E_1=H'E_\ell=1,\qquad H'E_i=0\quad(1<i<\ell).
\end{equation}
\end{lemma}
\begin{proof}
For a cycle, $S$ is connected of arithmetic genus one, so $h^1(\mathcal O_S)=1$. Rationality of $Y$ and
$0\to\mathcal O_Y(-S)\to\mathcal O_Y\to\mathcal O_S\to0$
give $h^2(\mathcal O_Y(-S))=1$. By Serre duality there is an effective divisor $A\sim K_Y+S$. But $A+H'\sim_\Q0$, and both are effective. Intersecting with an ample divisor forces $A=H'=0$, proving the assertion.
For a chain, adjunction gives $(\Gamma-E_i)E_i=2$. The other floor components in $S$ contribute one at an end and two at an interior vertex. Subtraction proves \eqref{eq:end-budget}.
\end{proof}

\subsection{Coregularity and reselection of the complement}
Recall that for an SNC surface pair, its \emph{dual complex} has a vertex for each coefficient-one component and an edge for each intersection of two such components. The dual complex of an lc pair is defined on a dlt model. Its dimension is a crepant birational invariant. The \emph{coregularity} of a surface log Calabi--Yau pair is $1-\dim\mathcal D$, with $\dim\varnothing=-1$. The \emph{absolute coregularity} of $X$ is the minimum obtained by varying its lc log Calabi--Yau boundaries; see \cite[Section~1]{FFMP}.

By \cref{prop:floor}, $X$ has absolute coregularity zero when $V<\hzero\eps^2$.
Apply \cite[Theorem~4]{FFMP} with zero initial boundary and zero nef part. A klt Fano surface is of Fano type, the zero boundary has standard coefficients, and twice the zero nef part is b-Cartier. We therefore obtain an lc boundary $B$ such that
\begin{equation}\label{eq:two-complement}
 2(K_X+B)\sim0,\qquad \coreg(X,B)=0.
\end{equation}
Note that this is a new choice of complement. All the preceding results can be applied to it. In particular, its crepant boundary on the minimal resolution is SNC and can be written
\begin{equation}\label{eq:half-boundary}
 \Gamma=S+\tfrac12H,
\end{equation}
where $S,H$ are reduced divisors without common components and $S$ is a connected chain or cycle with an edge.

\subsection{Removing positive Iitaka dimension cases}
The \emph{Iitaka dimension} $\kappa(G)$ of an effective $\Q$-divisor $G$ is the dimension of the maps given by sufficiently divisible multiples of $G$. We now apply the log surface MMP and abundance \cite[Theorems~3.3 and 6.1]{Fujino} to the reduced support of the complement. We give the argument for any bounded complement, since its numerical estimates follow without additional work.

\begin{proposition}\label{prop:kappa}
Suppose that the complement has the SNC configuration established above. Set $D=\Supp\Gamma$ and $G=D-\Gamma$, so that $K_Y+D\sim_\Q G$.
For $n>1$, the following estimates hold:
\begin{enumerate}
\item if $G$ is big, then $V\ge1/[462(n-1)^2]$;
\item if $\kappa(G)=1$, then $V\ge\eps/[n(n-1)]$.
\end{enumerate}
Consequently, for \eqref{eq:half-boundary} under $V<\hzero\eps^2$, either $H=0$ and $S$ is a reduced anticanonical cycle, or $S$ is a chain and the intersection matrix of $H$ is negative definite.
\end{proposition}
\begin{proof}
Suppose first that $G$ is big. The log canonical model $(Z,D_Z)$ of the reduced SNC pair $(Y,D)$ has $K_Z+D_Z$ ample. Its boundary is nonzero. Indeed, if $D_Z=0$, then the push-forward of $G$, which is supported on $D$, is zero, while $K_Z+D_Z\sim_\Q G_Z$, a contradiction.
Theorem~1.4 of \cite{LS} gives $\vol(K_Z+D_Z)\ge1/462$. Put $D_X=\Supp B$. Pushing rational sections down to the normal surface $X$ gives
\[
 \frac1{462}\le\vol(K_Y+D)
 \le\vol(K_X+D_X)=\vol(D_X-B)\le(n-1)^2V.
\]
For the middle inequality, an effective divisor $\operatorname{div}(\phi)+m(K_Y+D)$ pushes forward to an effective divisor with class $m(K_X+D_X)$. Thus we use only the divisorial conditions on rational sections, not log canonicity of $(X,D_X)$. The last inequality follows from
$0\le D_X-B\le(n-1)B$.

Now suppose $\kappa(G)=1$. Run the $(K_Y+D)$-MMP $h:Y\to Z$ and use abundance for surfaces. Every negative contraction is contained in the support of the current effective divisor $G$, since an effective divisor intersects non-negatively every curve outside its support. Consequently only fractional components are contracted. We have
\[
 G=h^*G_Z+N,\qquad N\ge0,
\]
with $N$ exceptional and $G_Z\sim_\Q K_Z+D_Z$ semi-ample of Iitaka dimension one. Let $\varphi:Z\to T$ be its connected Iitaka fibration. A divisible multiple of the effective divisor $G_Z$ is the zero divisor of a section pulled back from $T$, because $\varphi_*\mathcal O_Z=\mathcal O_T$. Hence $G_Z$ is a positive rational combination of complete fibres. The composite $g=\varphi h$ has a fibre $F_0$ with
\begin{equation}\label{eq:fractional-fibre}
 \Supp F_0\subseteq\Supp G.
\end{equation}

Every horizontal component of $\Gamma$ has coefficient one. Since $-K_Y=P+\Delta$ is big, its intersection with a general fibre $F$ is positive. Adjunction therefore gives $F\simeq\PP^1$ and $\Gamma F=2$. On horizontal floor components the coefficients of $U$ are either one or ambient log discrepancies at least $\eps$. Therefore
\begin{equation}\label{eq:PF}
 PF=UF\ge2\eps.
\end{equation}
The fibre \eqref{eq:fractional-fibre} has only fractional boundary components. Since the union of the boundary and that fibre is SNC,
\[
 \tau=\lct(Y,\Gamma;F_0)
 =\min_{E\subset F_0}\frac{1-\operatorname{coeff}_E\Gamma}
                            {\operatorname{mult}_EF_0}>0.
\]

Contract vertical $(-1)$-curves to obtain a smooth ruled surface. Each contraction is crepant for $K+\Gamma$: its exceptional difference is numerically trivial and hence zero. The surviving fibre over $F_0$ has coefficient $\beta\in(0,1)$ with $n\beta\in\Z$. Its horizontal boundary is reduced and has degree two. A singular horizontal branch, or an intersection of two horizontal branches, on this fibre would have horizontal multiplicity at least two. Together with $\beta>0$ this would violate lc at the first blowup. Thus there are two transverse intersections, or a smooth horizontal branch tangent to order two.
In the transverse case the threshold is $1-\beta$. In the tangent case adjunction to the coefficient-one branch, or the two successive blowups of the tangency, gives threshold $1/2-\beta$. For clarity, in local coordinates the tangent branch and fibre have equations $x=y^2$ and $x=0$; restriction of the fibre to the branch has multiplicity two, so the coefficient allowed by adjunction is $1/2$. The threshold is unchanged by the crepant contractions.

Its positivity and $n\beta\in\Z$ imply $\tau\ge1/(2n)$. On $Y$, this says $G\ge F_0/(2n)$. By nefness and \eqref{eq:PF},
\[
 PG\ge PF/(2n)\ge\eps/n.
\]
Exceptional components have $P$-degree zero. Thus
$PG=\sum_i(1-b_i)t_i\le(n-1)\sum_i b_it_i=(n-1)V$, proving (ii).

For $n=2$, both positive-dimensional alternatives contradict $V<\hzero\eps^2$. If the floor is a cycle, \cref{lem:cycle-floor} already gives $H=0$. Otherwise it is a chain. If $H\ne0$, then $G=H/2$ is effective of Iitaka dimension zero.
In the same MMP its semiample pushforward is numerically trivial, hence the effective divisor $G_Z$ is zero. Every component of $H$ is therefore contracted, and no component of $S$ is contracted. The exceptional intersection matrix is negative definite. Moreover
\begin{equation}\label{eq:half-crepant}
 K_Y+S+\tfrac12H=h^*(K_Z+S_Z),
\end{equation}
by negative definiteness applied to the exceptional numerically trivial difference. This completes the reduction.
\end{proof}

\section{The fractional configuration}\label{sec:fractional}

Throughout this section $\Gamma=S+H/2$ is as in \cref{prop:kappa}: $S$ is a chain of at least two floor components, $H$ has negative-definite intersection matrix, and the strict components have the contact description of \cref{thm:contact}. For simplicity, we write $2^\ell$ for $\ell$ consecutive entries equal to two in a chain word.

\subsection{The four attached arms}
\begin{lemma}\label{lem:arms}
Every connected component of $H$ has one of the words
\[
 [4],\qquad [3,2^\ell,3],\qquad [2],\qquad [1,2^\ell,3]
 \quad(\ell\ge0).
\]
The first two types are disjoint from $S$, consist entirely of $f$-exceptional curves, and have coefficient zero in $U$. Each of the last two types meets $S$ once, at an end of $S$, through the isolated $[2]$ or the $(-1)$-end. There are exactly four attached components, two at each end of $S$.
\end{lemma}
\begin{proof}
For a component $E$ of $H$, put $b_E=-E^2$, $d_E=\deg_H(E)$, and $s_E=ES$, counting intersection incidences. Adjunction yields
\begin{equation}\label{eq:arm-adj}
 b_E+d_E+2s_E=4.
\end{equation}
Note that an exceptional component has $b_E\ge2$. A strict component has self intersection $-1$ and exactly two exceptional contacts. Intersecting $K_Y+\Gamma$ with it shows that the coefficients at those contacts sum to $3/2$. Hence one is a floor contact and the other is a half-contact, so $s_E=d_E=1$.
Every $H$-vertex has degree at most two. A cycle would have all weights two and a null vector, contradicting negative definiteness. Thus each connected component is a path or an isolated vertex. An interior path vertex has $(b,s)=(2,0)$; an end has $(3,0)$ or $(1,1)$. An isolated vertex has $(4,0)$ or $(2,1)$. The possible word $[1,2^\ell,1]$ has a null vector and is excluded. This proves the list.

At an end of $S$, \eqref{eq:end-budget} gives $(H/2)E=1$. Thus there are two incidences with $H$ at each end, and no incidence at an interior floor component. Each attached path uses one incidence, giving exactly four arms.
An unattached component contains no strict curve. It is an entire uncontacted exceptional germ. To justify the latter assertion, if a non-canonical germ meets $\Supp\Gamma$, all its exceptional curves belong to that support because $\Delta$ is strictly positive there. For a canonical germ, a nonzero contact vector with $B$ gives strictly positive pullback coefficients on its entire resolution, by \cref{lem:positive}; without such a contact all these coefficients are zero. Hence no omitted zero-coefficient vertex can separate an unattached $H$-component from the rest of its germ. Since the germ is not contacted, $U=f^*B$ vanishes on it.
\end{proof}

\begin{lemma}\label{lem:arm-coeff}
Let $a$ be the coefficient of $U$ on an end of $S$. That end is exceptional and $a\ge\eps$. An isolated attached $(-2)$-curve has coefficient $a/2$.
For an attached arm
\[
 C,E_1,\ldots,E_{\ell+1}=[1,2^\ell,3],
 \qquad r=2\ell+3,
\]
the coefficients and the strict degree are
\begin{equation}\label{eq:arm-coeff}
 u_C=\tfrac12,\qquad
 u_{E_i}=\frac{r-2i}{2r},\qquad
 t_C=UC=a-\frac1r>0.
\end{equation}
In particular, the positive support of $U$ is the tree consisting of $S$ and the four attached arms of \cref{lem:arms}. If $V<\eps^2/16$, then $1/r\ge7a/8$, $r\le2/\eps$, and every positive coefficient of $U$ is at least $\eps/4$.
\end{lemma}
\begin{proof}
An end of $S$ has three incident boundary components. It cannot be strict, since a strict component has exactly two contacts. Thus its $U$-coefficient is its ambient discrepancy.
For an isolated arm, $UE=0$ gives $2u_E=a$. On a longer arm, $C$ is strict and has coefficient $1/2$. The exceptional equations $UE_i=0$ are
\[
 2u_{E_i}=u_{E_{i-1}}+u_{E_{i+1}},\qquad
 3u_{E_{\ell+1}}=u_{E_\ell},
\]
where $u_{E_0}=1/2$. Solving gives the displayed progression; for $\ell=0$, the single equation gives $u_{E_1}=1/6$. The tail $[2^\ell,3]$ has determinant $r$, by the continuant recurrence. Intersecting with $C$ gives
$UC=a-\frac12+\frac{r-2}{2r}=a-1/r$.

These formulas give positive coefficients on every attached arm. On a floor component, the coefficient of $U$ is either one or an ambient discrepancy at least $\eps$. Together with the vanishing statement in \cref{lem:arms}, this identifies the positive support of $U$.

Since $t_C\le2V$, the volume assumption gives
\[
 1/r=a-t_C\ge a-2V\ge7a/8\ge7\eps/8.
\]
Here $2V<\eps^2/8\le a/8$. The last tail coefficient is $1/(2r)\ge\eps/4$. The assertion for all other coefficients follows from $a\ge\eps$ and $u_C=1/2$.
\end{proof}

\subsection{An integral determinant estimate}
We prove the linear algebra result needed for the non-singular case. A \emph{spanning tree} of a connected finite graph is a connected subgraph containing all vertices and no cycle. For positive edge weights, its weight is the product of its edge weights.

\begin{lemma}\label{lem:matrix-tree}
Let $A$ be a non-singular real symmetric $N\times N$ matrix with non-negative off-diagonal entries and connected off-diagonal graph. Suppose
$t=A\one\ge0$, $v=\one^TA\one>0$, and $A$ has exactly one positive eigenvalue. Let $\tau$ be the sum of the weights of its spanning trees, using $A_{ij}$ as edge weights. Then
\[
 |\det A|\le v\tau.
\]
For $N=1$, the empty-tree weight is one.
\end{lemma}
\begin{proof}
First, we assume all $t_i>0$, and put $D_t=\diag(t_i)$ and $L=D_t-A$. Then $L$ is the weighted graph Laplacian: its off-diagonal entry is $-A_{ij}$ and each row sums to zero. It is positive semidefinite, since
$x^TLx=\sum_{i<j}A_{ij}(x_i-x_j)^2$.
The matrix $M=D_t^{-1/2}LD_t^{-1/2}$ has one-dimensional kernel spanned by $(\sqrt{t_i})$. Denote its other eigenvalues by $\lambda_2,\ldots,\lambda_N$.
The matrix $A$ is congruent to $I-M$. Nonsingularity and the inertia assumption imply $\lambda_j>1$ for all $j\ge2$. Hence
\begin{equation}\label{eq:det-eigen}
 |\det A|=\Bigl(\prod_i t_i\Bigr)\prod_{j=2}^N(\lambda_j-1)
 \le\Bigl(\prod_i t_i\Bigr)\prod_{j=2}^N\lambda_j.
\end{equation}

We recall the weighted matrix-tree identity with its proof. Orient the edges arbitrarily and form the incidence matrix, whose edge column has entries $1,-1$ at its ends and zero elsewhere. Multiplying the column of an edge by the square root of its weight gives a matrix $B$ with $L=BB^T$. Delete row $i$. Cauchy--Binet expresses $\det L^{(i)}$ as a sum over $N-1$ edge columns. A cycle gives dependent columns; a disconnected subgraph has deficient rank. For a spanning tree, successively deleting a leaf shows that the unweighted minor is $\pm1$. Thus every principal cofactor of $L$ equals $\tau$.

The $i$th principal cofactor of $M$ is consequently $\tau t_i/\prod_jt_j$. Its spectral expression is
$(\prod_{j=2}^N\lambda_j)t_i/v$, since the normalized kernel vector is $(\sqrt{t_i/v})$. Comparing these expressions in \eqref{eq:det-eigen} gives $|\det A|\le v\tau$.
If some $t_i=0$, replace $A$ by $A+\eta I$. For sufficiently small positive $\eta$, its inertia is unchanged, all row sums are positive, and the edge weights are unchanged. Let $\eta$ decrease to zero. This completes the proof.
\end{proof}

\begin{theorem}\label{thm:tree}
Let $D_i$ be curves on a smooth projective surface, with tree intersection graph and every edge intersection equal to one. Suppose $U=\sum_i u_iD_i$ is nef and big with $u_i>0$. If $I=(D_iD_j)$ is non-singular, then
\begin{equation}\label{eq:tree}
 U^2\ge|\det I|\prod_i u_i^{\,2-\deg(i)}
       \ge\prod_i u_i^{\,2-\deg(i)}.
\end{equation}
\end{theorem}
\begin{proof}
Apply \cref{lem:matrix-tree} to $A=\diag(u_i)I\diag(u_i)$. Its row sums are $u_i(UD_i)\ge0$, and their sum is $U^2>0$. By Hodge index theorem, we have exactly one positive eigenvalue, since the span of the curves contains $U$ and the form is negative definite on the numerical orthogonal complement of a big nef class.
A tree has exactly one spanning tree. Hence
\[
 \tau=\prod_{\{i,j\}\text{ edge}}u_iu_j
       =\prod_i u_i^{\deg(i)},\qquad
 |\det A|=|\det I|\prod_i u_i^2.
\]
Substitution gives the first inequality. The second uses the nonzero integral determinant $|\det I|\ge1$.
\end{proof}

\begin{proposition}\label{prop:fractional-nonsingular}
If the intersection matrix $I_U$ of the positive support of $U$ is nonsingular, then $V\ge\eps^2/64$.
\end{proposition}
\begin{proof}
We may assume $V<\eps^2/16$. Apply \cref{thm:tree} to the positive support described in \cref{lem:arm-coeff}. In \eqref{eq:tree}, degree-two vertices cancel. The two ends of $S$ are trivalent, and the four arm tips are leaves. Fix an end of $S$ with coefficient $a$, and let $j\in\{0,1,2\}$ count its attached strict half-components. The product of the two leaf coefficients, divided by $a$, is respectively
\[
 \frac a4,\qquad \frac1{4r},\qquad \frac1{4r_1r_2a}.
\]
By \cref{lem:arm-coeff}, these are at least $a/4$, $7a/32$, and $49a/256$. Each is at least $\eps/8$. The two ends therefore contribute at least $\eps^2/64$, as required.
\end{proof}

\subsection{A singular matrix gives a principal divisor}
We must treat a zero determinant separately. Order the floor components before the four arms and write
\[
 I_U=\begin{pmatrix}I_S&C\\ C^T&-M_H\end{pmatrix},
 \qquad J=I_S+CM_H^{-1}C^T.
\]
Here $M_H$ is positive definite, and $J$ is the Schur complement obtained by eliminating the arms.

\begin{lemma}\label{lem:J}
The matrix $J$ is integral and tridiagonal, with adjacent entries one. It has exactly one positive eigenvalue and nullity at most one. If $a$ is the vector of $U$-coefficients on $S$, then $a>0$, $Ja\ge0$, and its two endpoint entries are at least $\eps$.
\end{lemma}
\begin{proof}
The inverse attaching entry of $[2]$ is $1/2$. For $[1,2^\ell,3]$ the full determinant is two and the attaching cofactor is $r=2\ell+3$, so its inverse attaching entry is $r/2$. At each end of $S$ we add two odd halves to the diagonal, and thus an integer. No off-diagonal entry changes.
Hodge index gives one positive direction for $I_U$; numerical relations can only add zero directions. Eliminating a negative-definite block preserves the positive index. An eigenvector of a tridiagonal matrix with nonzero adjacent entries is determined by its first entry, which cannot vanish. Thus every eigenspace is one-dimensional and the nullity is at most one.
Writing $I_Uu=(d_S,d_H)^T\ge0$, we obtain
\[
 u_H=M_H^{-1}(C^Ta-d_H),\qquad
 Ja=d_S+CM_H^{-1}d_H\ge0.
\]
The inverse arm matrices are non-negative; this follows either from their explicit chain inverses or from the positive-inverse argument of \cref{lem:positive}. The endpoint assertion is \cref{lem:arm-coeff}.
\end{proof}

\begin{lemma}\label{lem:kernel}
Let $J$ be an integral tri-diagonal $s\times s$ matrix, $s\ge2$, with adjacent entries one and inertia $(1,s-2,1)$. If $a>0$ and $Ja\ge0$, its kernel has an integral generator $w$ with
\[
 w_1=1,\quad w_s=-1,\qquad
 \frac{w_1}{a_1}\ge\frac{w_2}{a_2}\ge\cdots\ge\frac{w_s}{a_s},
 \qquad |w_i|\le\frac{a_i}{\min(a_1,a_s)}.
\]
After zero entries are omitted, the signs change exactly once. There is at most one zero entry.
\end{lemma}
\begin{proof}
Normalize the first entry to one. The integral row recurrence makes all entries integral. Since the corank is one, $\adj(J)=\kappa ww^T$, where $\kappa$ is the first principal cofactor and is integral. The corner cofactor is $(-1)^{s+1}$: its minor is triangular with diagonal entries one. Thus $\kappa w_s=\pm1$ and $w_s=\pm1$.

The vector cannot have one sign. A non-negative kernel vector is positive by the recurrence. Conjugating $J$ by its diagonal matrix would then give a negative weighted path Laplacian, with no positive eigenvalue.
Partition its nonzero entries into maximal consecutive domains of constant sign, and let $v_j$ be the vector restricted to domain $j$. From $Jw=0$, for arbitrary scalars $c_j$ we compute
\[
 \Bigl(\sum_jc_jv_j\Bigr)^TJ\Bigl(\sum_jc_jv_j\Bigr)
 =\sum_{\substack{i,i+1\text{ in}\\\text{opposite sign domains}}}
 (-w_iw_{i+1})(c_{j(i)}-c_{j(i+1)})^2\ge0.
\]
Edges meeting a zero entry contribute zero. The domain vectors are independent, so they span a non-negative subspace. For inertia $(1,s-2,1)$ such a subspace has dimension at most two: projection to the positive eigenspace plus the kernel is injective on it. Thus there are exactly two sign domains. The first is positive and the last negative, giving $w_s=-1$. A zero entry has nonzero neighbours of opposite signs, and two consecutive zeros force the entire vector to vanish. Hence there is at most one zero.

Put $d=Ja\ge0$ and
$F_i=w_ia_{i+1}-w_{i+1}a_i$ for $1\le i<s$, with $F_0=F_s=0$. Subtracting the row recurrences gives
$F_i-F_{i-1}=w_id_i$.
These partial sums first increase and then decrease to zero, so $F_i\ge0$. Dividing by $a_ia_{i+1}$ gives the displayed monotonicity. Its endpoint values $1/a_1$ and $-1/a_s$ imply the absolute bound.
\end{proof}

\begin{proposition}\label{prop:principal}
If $V<\eps^2/16$ and $I_U$ is singular, there is a nonzero integral principal divisor $W$ on $Y$ with
\begin{equation}\label{eq:principal}
 \Supp W\subseteq\Supp U,\qquad |W|\le(4/\eps)U.
\end{equation}
Every arm coefficient is nonzero and has the sign of its attaching endpoint. The signs along $S$ change once.
\end{proposition}
\begin{proof}
By \cref{lem:J}, $J$ has inertia $(1,s-2,1)$. Extend the kernel vector of \cref{lem:kernel} to the arms by
$\widehat w=(w,M_H^{-1}C^Tw)$, so $I_U\widehat w=0$.
On $[2]$ its value is $w_{\rm end}/2$. On $[1,2^\ell,3]$ the values are
\[
 \widehat w_C=(r/2)w_{\rm end},\qquad
 \widehat w_{E_i}=\tfrac12(r-2i)w_{\rm end},
 \qquad w_{\rm end}\in\{1,-1\}.
\]
These are obtained by the same arm recurrence as in \cref{lem:arm-coeff}.
Let $W$ have coefficient vector $2\widehat w$. It is integral. On $S$ the coefficient ratio $|W_i|/u_i$ is at most $2/\eps$; on $[2]$ it is $2/a_{\rm end}$; and on a longer arm it is $2r\le4/\eps$. This proves \eqref{eq:principal}.
The kernel identity gives $WU=W^2=0$. Hodge index with $U$ nef and big therefore gives $W\equiv0$. A smooth rational surface has torsion-free Picard group and a nondegenerate intersection pairing, as seen in the blowup basis of $\PP^2$ or a Hirzebruch surface. Thus an integral numerically trivial divisor is linearly trivial. Hence $W$ is principal. Its opposite endpoint signs show that it is nonzero and has non-empty positive and negative parts.
\end{proof}

\begin{theorem}\label{thm:fractional}
The fractional configuration of this section satisfies $V\ge\eps^2/64$.
\end{theorem}
\begin{proof}
We may assume $V<\eps^2/16$. If $I_U$ is non-singular, use \cref{prop:fractional-nonsingular}. Otherwise,let $W=W_+-W_-$ be the principal divisor of \cref{prop:principal}. Its two parts define a pencil with no common components.
Their supports meet only at the unique opposite-sign edge of $S$, unless a zero component separates them, in which case there is no base point. At an opposite-sign node, local generators are $x^p$ and a unit times $y^q$.
The Euclidean algorithm resolves this pencil by node blowups: in one blowup chart the exponents become $(p-q,q)$ when $p>q$, and symmetrically $(p,q-p)$ when $q>p$; equality removes the remaining monomial base ideal. Multiplication by a unit does not change the ideal or these centres.
Let $\pi:\widetilde Y\to Y$ be the resulting resolution, and write
\[
 \widetilde U=\pi^*U,\qquad
 K_{\widetilde Y}+\widetilde\Gamma=\pi^*(K_Y+\Gamma).
\]
All centres are intersections of two floor components, so every newly extracted component has coefficient one in $\widetilde\Gamma$.

All attached half-components have nonzero order in $W$ and are vertical for the pencil. An unattached half-component is disjoint from $W_+\cup W_-$; the ratio defining the pencil restricts to an invertible regular function on that complete curve, and hence is constant. It too is vertical. Therefore all horizontal boundary components have coefficient one.
Pass to the Stein factorization, and let $F$ be a smooth connected general fibre. Hodge index gives $\widetilde UF>0$, since $\widetilde U$ is nef and big and $F$ is nonzero with square zero. The boundary has a horizontal component because $\widetilde U$ is supported on it. Adjunction yields
\[
 0=(K_{\widetilde Y}+\widetilde\Gamma)F
   =2g(F)-2+\widetilde\Gamma F.
\]
The last term is a positive integer. Thus $g(F)=0$ and $\widetilde\Gamma F=2$.

On a horizontal strict floor component the coefficient of $\widetilde U$ is one. On a horizontal exceptional floor component it is the ambient log discrepancy over $X$, at least $\eps$. This includes the newly extracted components, by the two crepant identities defining $\widetilde\Gamma$ and the ambient discrepancy divisor. Hence $\widetilde UF\ge2\eps$.
The total transform $\pi^*W_+$ is a fixed effective exceptional part plus a movable divisor of class $dF$, for an integer $d\ge1$. Pulling back \eqref{eq:principal} and intersecting with the nef divisor $\widetilde U$, we obtain
\[
 (4/\eps)V\ge\widetilde U\,\pi^*W_+
             \ge d\,\widetilde UF\ge2\eps.
\]
This gives $V\ge\eps^2/2$, contradicting $V<\eps^2/16$. The singular case is therefore excluded at the required scale.
\end{proof}

\section{Reduced anticanonical cycles}\label{sec:cycles}

We now consider the other configuration. Let $Y$ be a smooth projective rational surface and let $D$ be a reduced SNC anticanonical cycle with at least two smooth rational components. Two components in such a cycle meet twice; with at least three components, the graph is the usual simple cycle. Suppose
\begin{equation}\label{eq:cycle-setup}
 D=U+\Delta,\qquad U=\sum_i u_iD_i\text{ is nef and big},
 \qquad \eps\le u_i\le1,\qquad U\Delta=0.
\end{equation}
These are precisely the properties supplied by the reduced-complement case.

\subsection{A lattice estimate for a complete fibre}
\begin{lemma}\label{lem:half-plane}
Let a smooth integral fan subdivide the closed upper half-plane, with boundary rays $(-1,0)$ and $(1,0)$ and interior primitive rays $v_i=(p_i,m_i)$, $m_i>0$. Let $\psi$ be convex, positively homogeneous, and linear on each cone. Suppose its values on all primitive ray generators are at least $\eps>0$, and its boundary values are $a,b$. Put $A=\max(a,b)$. Then
\[
 \frac{m_i}{\psi(v_i)}\le\frac{4A}{\eps^2}.
\]
\end{lemma}
\begin{proof}
Since every cone is smooth, a lattice vector in it is a non-negative integral combination of its primitive generators. Thus every nonzero lattice vector in the half-plane has $\psi$-value at least $\eps$.
Fix $v_i=(p,m)$ and put $u=\psi(v_i)$. In coordinates
$v=t v_i+s(1,0)$, the inverse image of $\Z^2$ is a lattice of covolume $1/m$. The symmetric open rectangle
\[
 |t|<\frac{\eps}{2u},\qquad |s|<\frac{\eps}{2A}
\]
has area $\eps^2/(uA)$. If $m>4Au/\eps^2$, Minkowski's theorem gives a nonzero lattice point in this rectangle. Change its sign if necessary so that $t\ge0$; its image lies in the upper half-plane. Convex homogeneity implies subadditivity, so
\[
 \psi(v)\le t\psi(v_i)+\psi(s(1,0))
          \le tu+A|s|<\eps,
\]
a contradiction.
\end{proof}

\begin{proposition}\label{prop:complete-fibre}
In \eqref{eq:cycle-setup}, suppose a rational fibration with connected fibres has a complete fibre $F_0$ supported on a proper chain in $D$, with horizontal boundary branches at its two ends. Then $U^2\ge\eps^2/4$.
\end{proposition}
\begin{proof}
We may write $F_0=\sum_{i=1}^s m_iE_i$ and $E_i^2=-b_i$. A general fibre satisfies $DF=2$. The two end contacts are transverse, since $D$ is SNC. They contribute $m_1+m_s$ to the horizontal boundary degree, so $m_1=m_s=1$ and they exhaust that degree. For $s=1$ the same statement reads $2m_1=2$. The horizontal boundary is consequently either two sections or one bisection.
Let $a,b$ be the coefficients of $U$ on the two branches and let $u_i$ be its fibre coefficients. Put $A=\max(a,b)$.

Define lattice vectors
\[
 v_0=(-1,0),\qquad v_1=(0,1),\qquad
 v_{i+1}=b_iv_i-v_{i-1}.
\]
Their second coordinates are the fibre multiplicities, by $F_0E_i=0$. Adjacent determinants are $-1$, the interior second coordinates are positive, and the last one is zero. Since $m_s=1$, the final vector is $v_{s+1}=(1,0)$. The angles decrease strictly in the upper half-plane, so these rays form a smooth fan covering it once.
Assign the values $a,u_1,\ldots,u_s,b$ to its primitive rays. Nefness gives
\[
 u_{i-1}+u_{i+1}-b_iu_i=UE_i\ge0,
 \qquad u_0=a,\quad u_{s+1}=b.
\]
These are exactly the convexity inequalities for the homogeneous piecewise-linear function with those values: across a ray, the extension of either adjacent linear function is no greater than the value on the opposite neighbouring ray. Thus \cref{lem:half-plane} applies and gives the coefficientwise inequality
$U\ge\eps^2F_0/(4A)$.
Now $UF=a+b$, also when the two branches belong to one bisection and $a=b$. Therefore
\[
 U^2\ge\frac{\eps^2}{4A}UF
       =\frac{\eps^2(a+b)}{4A}\ge\frac{\eps^2}{4}.
\]
The horizontal boundary components need not survive on the original singular surface.
\end{proof}

\subsection{Contracting corners of a cycle}
A contraction of a boundary $(-1)$-curve in a cycle with at least three components will be called a \emph{corner contraction}. Its inverse is the blowup of a boundary node. We only use this terminology for smooth surfaces and reduced SNC cycles.

\begin{lemma}\label{lem:corner}
Suppose the intersection matrix of a reduced SNC anticanonical cycle $D$ has a positive direction. Then at least one of the following holds:
\begin{enumerate}
\item a component of $D$ has nonnegative self-intersection;
\item a rational fibration has a complete fibre supported on a proper chain in $D$, with horizontal boundary branches at its two ends;
\item corner contractions reduce $D$ to two curves meeting twice, with negative self-intersection word $(1,1)$, $(1,2)$, or $(1,3)$, up to exchange.
\end{enumerate}
\end{lemma}
\begin{proof}
Suppose all component squares are negative, and write $D_i^2=-b_i$, $b_i\ge1$. For a cycle with at least three components and all $b_i\ge2$, its intersection form is
\[
 x^TI_Dx=-\sum_i(x_i-x_{i+1})^2-\sum_i(b_i-2)x_i^2\le0,
\]
contrary to the assumption. Thus there is a boundary $(-1)$-curve. Its two neighbours are distinct; contracting it increases each neighbour's square by one and preserves a reduced SNC anticanonical cycle. A triangle becomes two smooth curves meeting transversely at two distinct points. Schur elimination removes a negative direction, so a positive direction remains.

Continue until a non-negative square first appears or only two components remain. In the first case that square is exactly zero. Let $C$ denote the corresponding curve on the current surface $Y'$. Adjunction gives $K_{Y'}C=-2$. Riemann--Roch gives $h^0(C)\ge2$, since $h^2(C)=h^0(K_{Y'}-C)=0$. Indeed, an effective $K_{Y'}-C$ would give an effective canonical divisor after adding $C$, which is impossible on a rational surface.
A second member of $|C|$ cannot contain $C$, since its effective difference with $C$ would be a nonzero numerically trivial divisor. Its intersection with $C$ is zero, so the two members are disjoint and generate a base-point-free pencil.
In the Stein factorization, $C$ is a whole fibre: it is reduced and irreducible, so its multiplicity and the number of fibres represented by that pencil member are both one. The general fibre is rational by adjunction.

Pull this fibration back to $Y$. Every inverse corner contraction is a boundary-node blowup; hence the total transform of $C$ is a complete fibre contained in $D$. Its support is proper, since $DF=2$ forces a horizontal boundary component. A proper connected subgraph of a cycle is a chain. Any outside boundary curve meeting it is horizontal, because distinct fibres are disjoint. This gives (ii).

If two negative components remain, their matrix is
$\left(\begin{smallmatrix}-a&2\\2&-b\end{smallmatrix}\right)$.
A positive eigenvalue implies $ab<4$. The positive integer possibilities are exactly $(1,1),(1,2),(1,3)$ up to exchange.
\end{proof}

\subsection{The toric comparison}
We record the toric estimate solely as an ingredient in the comparison construction. Areas below are normalized so that a fundamental lattice parallelogram has area equals one.

\begin{lemma}\label{lem:toric}
An $\eps$-lc toric weak del Pezzo surface has anticanonical volume at least $\eps^2$.
\end{lemma}
\begin{proof}
We may pass to the toric anticanonical model. Let $P_0$ be the convex hull of its primitive fan-ray generators in the lattice $N$, and $Q=P_0^\circ$ its polar in the dual space. The toric divisor-polytope and degree formulas \cite[Sections~3.4 and 5.3]{Fulton} give
$V=2\area(Q)$.
The homogeneous discrepancy function has value one at each primitive fan ray and has $P_0$ as its unit sublevel set. The $\eps$-lc condition and the toric discrepancy formula \cite[Section~1]{AmbroMld} imply
\[
 \operatorname{int}(\eps P_0)\cap N=\{0\}.
\]
Put $C=\conv(Q\cup(-Q))$. Then $C$ is centrally symmetric and
$C^\circ=P_0\cap(-P_0)$. Minkowski's theorem gives
$\area(C^\circ)\le4\eps^{-2}$.
The planar symmetric Mahler inequality \cite{Mahler} and the symmetric-hull inequality of Rogers--Shephard \cite{RS}, also stated in \cite[equation~(2)]{BC}, give
\[
 \area(C)\area(C^\circ)\ge8,\qquad
 \area(C)\le4\area(Q).
\]
The latter uses $0\in Q$. Combining the three inequalities,
\[
 V=2\area(Q)\ge\tfrac12\area(C)
       \ge4/\area(C^\circ)\ge\eps^2.
\]
\end{proof}

\begin{theorem}\label{thm:cycle}
Under \eqref{eq:cycle-setup}, $U^2\ge\eps^2/6$.
\end{theorem}
\begin{proof}
Since $U^2>0$, the boundary matrix has a positive direction. Apply \cref{lem:corner}.
In case (i), a prime component $C$ with $C^2\ge0$ is nef. The coefficient-wise inequalities $U\ge\eps D\ge\eps C$ and adjunction give
\[
 U^2\ge\eps UC\ge\eps^2DC
       =\eps^2(2+C^2)\ge2\eps^2.
\]
Case (ii) is \cref{prop:complete-fibre}.

In case (iii), let $(1,k)$ be the terminal word, where $k=1,2,3$, and put $r=3,4,6$, respectively. Repeating this word $r$ times is the boundary word of a smooth complete toric surface. Here are explicit primitive rays: take the listed vectors, followed in order by their negatives.
\begin{center}
\begin{tabular}{ccl}
\toprule
$k$&$r$&First half of the cyclic ray list\\
\midrule
$1$&$3$&$(1,0),(0,1),(-1,1)$\\
$2$&$4$&$(1,0),(0,1),(-1,1),(-2,1)$\\
$3$&$6$&$(1,0),(0,1),(-1,1),(-3,2),(-2,1),(-3,1)$\\
\bottomrule
\end{tabular}
\end{center}
Adjacent determinants are one, and the rays wind once around the origin. The relations $v_{i-1}+v_{i+1}=b_iv_i$ give $(1,k)$ repeated $r$ times, up to cyclic shift.
Reverse each original corner contraction in all $r$ copies, by the corresponding toric star subdivisions. For a two-component cycle the two intersection points are treated as distinct incidences; choose in each copy the incidence at which the original inverse blowup occurs. This produces a smooth complete toric surface $\widehat Y$ whose boundary word is the original word of $D$ repeated $r$ times.

Repeat the original coefficients $u_i$ to define $\widehat U$, and put
$\widehat\Delta=\widehat D-\widehat U$. The intersection rows agree with the original rows, including both contacts in a two-component cycle. Consequently $\widehat U$ is nonnegative on every invariant curve, hence nef by the toric nefness criterion \cite[Section~3.4]{Fulton}. Also
\begin{equation}\label{eq:repeat}
 \widehat U^2=rU^2>0,\qquad
 \widehat U\widehat\Delta=rU\Delta=0,\qquad
 0\le\operatorname{coeff}_i\widehat\Delta\le1-\eps.
\end{equation}
The toric nef and big divisor $\widehat U$ is semi-ample. Let $q:\widehat Y\to Z$ be its contraction. Every component of $\widehat\Delta$ is contracted, since its nonnegative $\widehat U$-degree contributes to the zero sum in \eqref{eq:repeat}. All uncontracted boundary components therefore have coefficient one in $\widehat U$.
With the toric canonical divisors, $q_*\widehat U=-K_Z$. Toric surfaces are $\Q$-factorial, and this divisor is ample because a divisible multiple is the ample divisor defining the contraction. The exceptional difference
$K_{\widehat Y}+\widehat\Delta-q^*K_Z$
is numerically trivial over $Z$, so it vanishes by negative definiteness.

The effective SNC pair $(\widehat Y,\widehat\Delta)$ is $\eps$-lc: its component discrepancies are at least $\eps$, and a further point blowup gives the sum of two component discrepancies, one plus a component discrepancy, or two off the boundary. Thus $Z$ is an $\eps$-lc toric Fano surface of volume $rU^2$. By \cref{lem:toric}, $rU^2\ge\eps^2$. Since $r\le6$, the theorem follows.
\end{proof}

\subsection{Proof of the volume theorem}
\begin{proof}[Proof of the lower bounds in \cref{thm:main}]
By \cref{lem:ac}, we may assume $X$ is Fano. If it is exceptional, \cref{prop:complement} gives $V\ge\cexc$.
Otherwise choose a bounded complement and suppose, for a contradiction, that $V<\hzero\eps^2$. By \cref{prop:floor}, its floor on the minimal resolution is a connected chain or cycle with an edge. The coregularity-zero theorem then gives the $2$-complement \eqref{eq:two-complement}. Reapplying the proved contact and boundary results to this complement gives \eqref{eq:half-boundary}.

By \cref{prop:kappa}, either the fractional part is nonzero with negative-definite intersection matrix and chain floor, or the boundary is a reduced anticanonical cycle. The first case gives $V\ge\eps^2/64$ by \cref{thm:fractional}, a contradiction.
In the second case, $D=\Gamma=U+\Delta$ satisfies \eqref{eq:cycle-setup}: on a strict component the coefficient of $U$ is one, and on an exceptional component it is its ambient discrepancy, in $[\eps,1]$. Moreover $U\Delta=0$ because $U=f^*B$ and $\Delta$ is exceptional. Hence \cref{thm:cycle} gives $V\ge\eps^2/6$, also a contradiction.
This proves all the lower-bound assertions.
\end{proof}

\begin{remark}\label{rem:nonvanishing}
There is a complementary route from anticanonical nonvanishing to volume. In the present setting, \cite[Corollary~1.7]{Zhu}, applied to the integral nef and big Weil divisor $-lK_X$ and the effective divisor zero, gives
\[
 h^0(X,-lK_X)\ge2\quad\Longrightarrow\quad
 V(X)\ge\frac{\eps}{2l^2}.
\]
For a Fano surface whose singularities are all scalar cyclic quotients $\frac1r(1,1)$, Lin's Riemann--Roch argument gives such an $l\le5$ \cite[Theorem~5.2]{Lin}, and hence $V\ge\eps/50$. We do not use this approach in the proof above. The quadratic examples below show that $l$ cannot be bounded absolutely for all cyclic quotient Fano surfaces.
\end{remark}

\section{Examples and optimal exponents}\label{sec:examples}

We will give three geometric constructions. The first forces the quadratic exponent for volume. The second realizes a cubic cyclic determinant on a Fano surface. The third realizes cubic Picard growth for smooth rational log Calabi--Yau pairs. For comparison of exponents we use the actual discrepancy threshold
\[
 e(X):=\min\bigl(\{1\}\cup\{\mld_p(X):p\in\operatorname{Sing}X\}\bigr).
\]

\subsection{Quadratic volume decay}
\begin{proposition}\label{prop:quadratic-example}
Let $u\ge2$ and $r=u^2+u+1$. Let a cyclic group of order $r$ act on $\PP^2$ by
\[
 [x:y:z]\longmapsto[x:\zeta y:\zeta^{u+1}z],
\]
where $\zeta$ is a primitive $r$th root of unity. Its quotient $T_u$ is a Fano surface with
\[
 \rho(T_u)=1,\qquad e(T_u)=\eps_u=\frac{u+2}{r},
 \qquad V(T_u)=\frac9r,\qquad
 \frac{V(T_u)}{\eps_u^2}=\frac{9r}{(u+2)^2}\longrightarrow9.
\]
\end{proposition}
\begin{proof}
The differences of the diagonal weights are $1,u,u+1$, all coprime to $r$. Thus every nonidentity element has three distinct eigenvalues and fixes only the coordinate points. The quotient map is quasi-\'etale, meaning finite and unramified in codimension one. Canonical pullback \cite[Proposition~5.20]{KM} gives
$K_{\PP^2}=\pi^*K_{T_u}$, and hence $V(T_u)=9/r$. Ampleness descends under the finite map. Pullback injects the numerical divisor space of $T_u$ into that of $\PP^2$, so $\rho(T_u)=1$.

The local weights at the three coordinate points are
$(1,u+1)$, $(-1,u)$, and $(-(u+1),-u)$. Normalizing and interchanging coordinates identifies each with $\frac1r(1,u+1)$, using $u(u+1)\equiv-1\pmod r$. Its negative continued fraction and discrepancies are
\[
 \frac r{u+1}=[u+1,2^u],\qquad
 \alpha_i=\frac{u+2+(i-1)(u-1)}r
 \quad(1\le i\le u+1).
\]
The continuant and discrepancy recurrences verify these formulas, including the endpoint conditions $\alpha_0=\alpha_{u+2}=1$. The minimum is $(u+2)/r$.
Further point blowups do not decrease that minimum: their discrepancies are sums of two old discrepancies, one plus an old discrepancy, or two. Thus $e(T_u)=\eps_u\to0$.
The displayed limit follows. For $p<2$ we have
$V(T_u)/\eps_u^p=(V(T_u)/\eps_u^2)\eps_u^{2-p}\to0$, proving the optimality assertion in \cref{thm:main}.
\end{proof}

\subsection{A common ruled-surface construction}
Fix $N\ge3$, set $e=1/N$, and start with $S_0=\PP^1_z\times\PP^1_y$, ruled by $z$. Let $S$ be a section class and $F$ a fibre class, so $S^2=F^2=0$ and $SF=1$. Choose:
\begin{itemize}
\item the fibre $V_0=(z=0)$ and $p=(0,0)$;
\item the smooth rational curve $C_0\in|NS+F|$ given by $z=y^N$;
\item a general $D_0\in|S+F|$ through $p$, transverse to $V_0$ there;
\item a general section $A_0\in|S|$ avoiding $p$, and a general fibre $V_1\ne V_0$.
\end{itemize}
We choose the curves to avoid unwanted triple intersections and make all intersections away from $p$ transverse. These are open conditions in the indicated linear systems. The curve $D_0$ meets $C_0$ in $N$ distinct transverse points away from $p$, since $D_0C_0=N+1$ and the intersection at $p$ is one.
Set
\begin{equation}\label{eq:example-boundary}
 \Theta_0=(1-e)(V_0+D_0)+e(A_0+C_0+V_1).
\end{equation}
The total section coefficient is $(1-e)+e+Ne=2$, and the fibre coefficient is $2(1-e)+2e=2$. Thus $K_{S_0}+\Theta_0\sim_\Q0$.

Blow up $p$, and then repeatedly blow up the point where the strict transform of $C_0$ passes through the node between the strict transform of $V_0$ and the newest exceptional curve. There are $N$ blowups in total. Indeed, after $j<N$ steps, local coordinates have
\[
 V_0=(x=0),\qquad E_j=(y=0),\qquad C_0=(x=y^{N-j}).
\]
The next blowup reduces the exponent by one. After the $N$th step, $C_0$ meets the newest exceptional curve away from its fibre nodes.

Let $S_N$ be the resulting surface and $E_1,\ldots,E_N$ the successive exceptional curves. We keep the same names for all strict transforms. The first exceptional discrepancy is $e$, because the total boundary multiplicity at $p$ is $2-e$. Each subsequent centre has vertical discrepancies $e,e$ and horizontal multiplicity $e$, so its new discrepancy is $e+e-e=e$. The reduced fibre over zero is
\[
 V_0,E_N,E_{N-1},\ldots,E_1.
\]
Every component of this fibre has discrepancy $e$, and $D_0$ meets $E_1$ at the other end.
The crepant boundary on $S_N$ is SNC, with coefficients $1-e$ on $V_0,D_0,E_1,\ldots,E_N$ and $e$ on $A_0,C_0,V_1$. It is therefore $e$-lc. Crepancy also verifies $e$-lc of the original pair.

\subsection{Cubic determinant mass on Fano surfaces}
\begin{proposition}\label{prop:cubic-example}
For every $N\ge3$ there is a $1/N$-lc Fano surface $X_N$ with one noncanonical singular point $p$ and possibly additional canonical points, such that
\[
 D_p=N^3+N^2+2N,\qquad
 \ind_p(K_{X_N})=N,\qquad V(X_N)=2/N.
\]
Its minimal resolution has Picard number $2N+6$. In particular, the exponent three in the determinant-mass estimate cannot be decreased.
\end{proposition}
\begin{proof}
Continue the construction on $S_N$ by blowing up all intersections between
$V_0+D_0+E_1+\cdots+E_N$ and $A_0+C_0+V_1$.
There are $N+4$ distinct transverse centres: one on $V_0\cap A_0$, one on $E_N\cap C_0$, and $N+2$ on $D_0\cap(A_0+C_0+V_1)$. The coefficients at each centre sum to one, so each new exceptional coefficient is zero.
Let $Y_N$ be the resulting smooth surface and put
\[
 \Delta_N=(1-e)(V_0+D_0+E_1+\cdots+E_N),\qquad
 P_N=e(A_0+C_0+V_1).
\]
Then
\begin{equation}\label{eq:example-nef}
 K_{Y_N}+\Delta_N+P_N\sim_\Q0,\qquad
 \Supp\Delta_N\cap\Supp P_N=\varnothing.
\end{equation}
The low-coefficient curves have squares
$A_0^2=-2$, $C_0^2=-1$, and $V_1^2=-1$.
For example, $C_0$ starts with square $2N$, loses $N$ in the first sequence, and loses $N+1$ at its intersections with $D_0$ and $E_N$. Their mutual intersections remain
$A_0C_0=1$, $A_0V_1=1$, and $C_0V_1=N$. Hence
\[
 P_NA_0=0,\qquad P_NC_0=P_NV_1=1,\qquad
 P_N^2=e^2(2N)=2/N>0.
\]
Every curve outside $\Supp P_N$ has nonnegative intersection with $P_N$. Thus $P_N$ is nef and big.

The pair $(Y_N,\Delta_N)$ is klt. A sufficiently divisible multiple of
$P_N\sim_\Q-(K_{Y_N}+\Delta_N)$ is base-point-free by \cite[Theorem~3.3]{KM}; denote its birational contraction by $g:Y_N\to X_N$. Every component of $\Delta_N$ is contracted, by \eqref{eq:example-nef}.
If $mP_N=g^*L$ with $L$ ample Cartier, pushing forward gives $L\sim-mK_{X_N}$. Thus $K_{X_N}$ is $\Q$-Cartier and $-K_{X_N}$ is ample. The divisor
$K_{Y_N}+\Delta_N-g^*K_{X_N}$ is exceptional and numerically trivial, hence zero. Consequently $X_N$ is $e$-lc and has volume $2/N$.

We verify that $Y_N$ is already the minimal resolution. The connected support of $\Delta_N$ is the chain
\begin{equation}\label{eq:example-chain}
 [N+1,2^N,N+1].
\end{equation}
Indeed $V_0$ has square $-N$ after the first $N$ blowups and then loses one more. The square of $D_0$ is $2-1-(N+2)=-N-1$. Each $E_j$ ends with square $-2$, including $E_N$ after its last centre on $C_0$.
If $T$ is any other contracted curve, Hodge index gives $T^2<0$ and adjunction gives
\[
 2p_a(T)-2=T^2-\Delta_NT<0.
\]
Thus $T$ is smooth rational. The only possibilities are $T^2=-2$, $\Delta_NT=0$, or $T^2=-1$, $\Delta_NT=1$. The latter would require $(1-1/N)m=1$ for an integer $m\ge0$, impossible for $N\ge3$. Hence all additional contracted curves are $(-2)$-curves disjoint from \eqref{eq:example-chain}. There is no exceptional $(-1)$-curve, and the other singularities are canonical. The rank is
$\rho(Y_N)=2+N+(N+4)=2N+6$.

Every discrepancy on \eqref{eq:example-chain} is $1/N$, as follows directly from \eqref{eq:disc-recurrence}. The determinant identity gives
\[
 D_p=2N+(N+1)N^2=N^3+N^2+2N.
\]
If $r=D_p$ and $q$ is the determinant after deleting an endpoint, then $(q+1)/r=1/N$, so $q=N^2+N+1$. The quotient is $\frac1r(1,q)$ up to reversal. Its canonical index is $r/\gcd(r,q+1)$: the character on $dx\wedge dy$ is $1+q$ \cite[Sections~2--3]{Reid}. Since $r=N(q+1)$, the index is exactly $N$.
\end{proof}

\begin{remark}
The last example has only two exceptional curves of square less than $-2$, but its cyclic determinant is cubic. The ratio of that determinant to the canonical index is $N^2+N+2$. Thus neither a quadratic determinant bound nor an identification of determinant with canonical index could replace the global argument in this paper.
\end{remark}

\subsection{Cubic Picard growth for rational log Calabi--Yau pairs}
\begin{proposition}\label{prop:picard-example}
There are smooth rational $1/N$-lc log Calabi--Yau pairs $(\Sigma_N,\Theta_N)$ with
\[
 \rho(\Sigma_N)\ge2+N+\frac N4\left\lfloor\frac N2\right\rfloor^2.
\]
Thus the cubic exponent for general smooth rational log Calabi--Yau pairs in \cref{thm:cubic} is optimal.
\end{proposition}
\begin{proof}
Return to $S_N$ before the separation blowups of \cref{prop:cubic-example}. Its fibre over zero has $N$ nodes, each with discrepancies $1/N,1/N$, and its horizontal boundary avoids those nodes.
At each node, perform toric point blowups inserting every primitive ray $(p,q)$ with $p,q\ge1$ and $p+q\le N$. The discrepancy of this ray is $(p+q)/N$, by repeated application of the node formula with zero horizontal multiplicity. Every new boundary coefficient is in $[0,1-1/N]$.

These subdivisions can be performed by ordinary point blowups while keeping the fan smooth. For adjacent primitive generators $u,v$, a blowup inserts $u+v$. A desired primitive vector $pu+qv$ with $p>q$ belongs to the cone $(u,u+v)$ with coordinates $(p-q,q)$; for $q>p$ use the cone $(u+v,v)$ and coordinates $(p,q-p)$. Induction on $p+q$ terminates, with $p=q=1$ as the primitive equal-coordinate case. Every auxiliary inserted ray has coordinate sum at most that of the desired ray. Repeating for the finitely many desired rays therefore gives precisely an allowed smooth subdivision.

The operations are crepant, and the boundary remains effective and SNC. Thus the final pair is $1/N$-lc with $K_{\Sigma_N}+\Theta_N\sim_\Q0$. The $N$ original nodes have disjoint neighbourhoods, so the counts add. If $A_N$ counts positive coprime pairs with $p+q\le N$, then
$\rho(\Sigma_N)=2+N+NA_N$.
Put $m=\lfloor N/2\rfloor$. The square $1\le p,q\le m$ lies in this triangle. Its noncoprime pairs number at most
\[
 \sum_{d=2}^m\lfloor m/d\rfloor^2
 \le m^2\sum_{d=2}^{\infty}d^{-2}
 \le m^2\left(\frac14+\int_2^\infty x^{-2}\,dx\right)
 =\frac34m^2.
\]
Therefore $A_N\ge m^2/4$, proving the lower bound. The matching upper order is \cref{prop:potential}.
\end{proof}

The upper bounds in \cref{thm:cubic} were proved in Section~\ref{sec:fibres}. \Cref{prop:quadratic-example,prop:cubic-example,prop:picard-example} prove all the optimality assertions in \cref{thm:main,thm:cubic}, completing the proofs of both main theorems.

\end{document}